%% file: main.tex
\documentclass[11pt, a4paper, reqno]{amsart}
\input{macros}

\begin{document}

\title{Elliptic Curves in Game Theory}

\author{Abhiram Kidambi}
\address{Max Planck Institute for Mathematics in the Sciences, Inselstra\ss e 22, 04103 Leipzig, Germany}
\email{kidambi@mis.mpg.de}

\author{Elke Neuhaus}
\address{Max Planck Institute for Mathematics in the Sciences, Inselstra\ss e 22, 04103 Leipzig, Germany}
\email{elke.neuhaus@mis.mpg.de}

\author{Irem Portakal}
\address{Max Planck Institute for Mathematics in the Sciences, Inselstra\ss e 22, 04103 Leipzig, Germany}
\email{mail@irem-portakal.de}

\begin{abstract}
   We investigate Spohn curves, the algebro-geometric models of totally mixed dependency equilibria for $2 \times 2$ normal-form games. These curves arise as the intersection of two quadrics in $\mathbb{P}^3$ and are generically elliptic curves. We examine the reduction of Spohn curves to plane curves, providing a full classification of conditions under which they are reducible. Notably, we prove that the real points are dense on the Spohn curve in all cases, which is relevant for applications. These computations are further supported by Macaulay2 and stored in Mathrepo. We review methods to compute the $j$-invariants of elliptic curves arising as the intersection of quadrics in $\mathbb P^3$ which we apply to the case of Spohn curves {aimed at game theorists}. We propose a definition of equivalence of generic $2\times 2$ games based on the $j$-invariant of the Spohn curve.
\end{abstract}

\maketitle

\tableofcontents

\section{Introduction}
\blfootnote{Acknowledgments: We thank Daniel Windisch, Lorenzo Baldi, Bernd Sturmfels and Máté Telek for useful conversations.}
In game theory, one is usually interested in the best possible solution of a game. More precisely, one wants to reach an equilibrium such that the choice of strategy of each player is optimal for them. 
The concept of \emph{Nash equilibria} fulfills this by giving strategies in which, if the players interact independently of each other, it makes no rational sense for any of them to deviate from the joint strategy while the others do not. Due to the assumption of independence, though, often these do not give desirable or ideal solutions. Wolfgang Spohn proposed a different concept of equilibria, the so-called \emph{dependency equilibria}, in which it is assumed that the players take into account their mutual expectations and try to maximize their conditional expected payoff (\cite{Spohn2},\cite{Spohn1}). 

This notion of equilibrium was first studied from an algebro-geometric point of view in \cite{PortakalSturmfels} and \cite{PortakalWindisch}, which characterize it, up to a certain point, via the so-called \emph{Spohn variety} $\V$. For generic games, the dimension and degree of the Spohn variety are known, as well as the fact that indeed all totally mixed Nash equilibria lie on the intersection of the open probability simplex, the Segre variety and $\V$. Moreover, it is proven that for generic games that any Nash equilibrium is a dependency equilibrium \cite[Corollary 3.20]{PortakalWindisch}. 

For $2\times 2$ games, i.e.\ games with two players in which both only have two choices, the Spohn variety generically takes the form of an \emph{elliptic curve} that is the intersection of two quadrics in $\mathbb P^3$, which can be reduced to a cubic in $\mathbb P^2$ resulting in the more familiar definition of the elliptic curve.  Namely, an elliptic curve over a field $K$ (which we take to be $\mathbb Q$) is a smooth projective curve of genus one with a marked point (at infinity). {If two elliptic curves over $\mathbb Q$ are isomorphic to each other, then they have same \emph{$j$-invariant}.  Determining isomorphisms of cubics in $\mathbb P^2$ directly can be computationally simplified by first computing the $j$-invariant and only checking for isomorphisms if the $j$-invariants are equal.
Clearly, dependency equilibria pose more of a computational challenge than Nash equilibria, where two-player games can be formulated as a linear algebra problem and they are generically zero-dimensional. Moreover, the case of $2\times2$ games constitutes a special case that must be treated separately, since for other finite generic games, the Spohn variety is rational \cite[Theorem 3.4]{PortakalSturmfels}. From an application-related point of view, of course, the interest lies in actual real solutions. This has only been addressed for generic games in \cite{PortakalWindisch} before. Non-generically, the question of the denseness of the real points inside the Spohn variety remains open. In this paper, we thoroughly examine all $2\times 2$ games and provide explicit computations. \\ 

This paper is structured as follows: In Section~\ref{sec:depeq}, relevant background on dependency equilibria and the Spohn variety is summarized. We prove in Proposition~\ref{prop: NE are DE} that for all $2 \times 2$ games, every Nash equilibrium is indeed a dependency equilibrium. We prove for prisoners' dilemma type of games, the unique Nash equilibrium is Pareto dominated by infinitely many dependency equilibria of the game (Proposition~\ref{prop: prisoners dilemma type}). Additionally, in Remark~\ref{rem: Pareto dom} we introduce a computational method to check whether the set of totally mixed Nash equilibria of a generic $2\times 2$ game is Pareto dominated by a dependency equilibrium. 

In Section~\ref{sec:denseness}, we investigate when the planar model of such a curve coming from a $2\times 2$ game is reducible (Theorem~\ref{thm: cases}), using properties of planar cubics and computations in Macaulay2 \cite{M2}. This happens in precisely 12 cases with an interesting combinatorial structure, which is explained in Remark~\ref{rem: combinatorialstructure}. We first use this to prove the denseness of real points for the planar cubic in Lemma~\ref{lem: denseness for plane cubic}, namely by finding a real smooth point in each irreducible component. Then, in Theorem~\ref{smooth real point in V}, we are able to pull back most of said points to components of $\V$ and deduce the same result for the Spohn variety in $\mathbb P^3$ in the aforementioned 12 cases. This is a first step towards the denseness of real points in the Spohn variety of non-generic games as addressed above. 

Section~\ref{sec:equivalentgames} presents a detailed overview of known results on elliptic curves and applies them to our case, i.e.\ to generic $2\times 2$-games. The goal of this section is to provide readers, especially with game theoretic backgrounds, with the necessary and sufficient established computational methods in the arithmetic of elliptic curves. We review an algorithm and code to compute the $j$-invariant of an elliptic curve that arises from the intersection of two quadrics in $\mathbb P^3$ with a known common rational point. This can be applied to elliptic curves which arise as the Spohn variety. We propose a game-theoretic equivalence between $2\times 2$ games based on the $j$-invariants of elliptic curves. In this section, we work with the cases when the payoff matrices, and thus the Spohn curve, are defined over $\mathbb Q$. In the case of real payoff matrices, Appendix~\ref{app:contfrac} explains how to obtain rational approximations to real numbers. 
\\[0.2cm]
All relevant code for this paper is hosted on the MathRepo page:
\begin{center}
   \begin{mdframed}\centering \url{https://mathrepo.mis.mpg.de/elliptic_curves_game_theory/}
   \end{mdframed}
\end{center}

\section{Dependency equilibria and the Spohn variety}
\label{sec:depeq}
We consider a normal-form game with $n$ players $1, \ldots,n$, where each player $i$ can choose from $d_i$ pure strategies $1, \ldots, d_i$. The outcome of the game depends on the choices of the players and is represented by the payoff tables $X^{(1)}, \ldots, X^{(n)}$. These are tensors of format $d_1 \times \ldots \times d_n$, such that, for each player $i$, the entry $X^{(i)}_{j_i \cdots j_n} \in \mathbb R$ of the payoff table $X^{(i)}$ specifies their payoff in the case that every player $l$ chooses pure strategy $j_l$. Formally, the game is then denoted by $X$ and is said to be a $(d_1 \times \ldots \times d_n)$-game in normal-form.
We may understand the players in such a game as probability variables with state space $[d_i]$ that decide the joint outcome. The probabilities of the joint decisions are recorded in the $d_1 \times \ldots \times d_n$-format tensor $p$. Its entries $p_{j_i \cdots j_n}$ are the probabilities that every player $l$ chooses the strategy $j_l$. Coming from an algebraic perspective, we view $p$ as an element in the projective space $\mathbb P^{d_1 \cdots d_n -1}$ over $\mathbb C$. By definition, $p$ must have non-negative real entries that sum up to $1$ and therefore lives in the projectivization $\overline{\Delta} \subset \mathbb P^{d_1 \cdots d_n -1}$ of the ($d_1\cdots d_n -1$)-dimensional probability simplex $\Delta_{d_1 \cdots d_n-1}$. We also define the open simplex $\Delta:=\Delta_{d_1\cdots d_n}^{\circ} \subset \mathbb P^{d_1 \cdots d_n -1}$ of probability tensors with nonzero entries.
\\[0.4cm]\noindent
The \emph{conditional expected payoff} of the $i$th player, conditioned on them choosing a certain pure strategy $k$ with respect to $p\in \overline{\Delta}$, is the sum

$$
\mathbb{E}^{(i)}_k (p) := 
 \sum\limits_{j_1=1}^{d_1} \cdots \widehat{\sum\limits_{j_i=1}^{d_i}} \cdots \sum\limits_{j_n=1}^{d_n} X^{(i)}_{j_1\cdots k \cdots j_n} \frac{p_{j_1\cdots k \cdots j_n}}{p_{+\ldots+k+ \ldots +}}.
$$
Here,
$$p_{+\ldots+k+ \ldots +}:= 
 \sum\limits_{j_1=1}^{d_1} \cdots \widehat{\sum\limits_{j_i=1}^{d_i}} \cdots \sum\limits_{j_n=1}^{d_n} p_{j_1\cdots k \cdots j_n},$$
with $k$ in the $i$-th position, is the probability that player $i$ actually chooses the pure strategy $k$.

While for Nash equilibria, players maximize their expected payoff, dependency equilibria are defined by players maximizing their conditional expected payoff. This allows for some form of commitment between the players, as discussed by Spohn in \cite[§2]{Spohn2}.

\begin{definition}[Dependency equilibrium, \cite{Spohn1}] \label{DESpohn}  
    A \emph{dependency equilibrium} is a joint probability distribution $p \in \overline{\Delta}$ that is the limit of a sequence $(p^{(r)})_{r \in \mathbb N}$ in $\Delta$ satisfying
    $$ \lim\limits_{r \to \infty} \mathbb E^{(i)}_k (p^{(r)}) \geq 
    \lim\limits_{r \to \infty} \mathbb E^{(i)}_{k'} (p^{(r)})   $$
    for all players $i \in \{ 1, \ldots n \}$ and all pure strategies $k,k' \in \{ 1, \ldots, d_i \}$ of player $i$ with $p_{+\ldots+k+ \ldots +}\neq~0$.   
\end{definition}

Defining dependency equilibria via a limit for some cases is necessary, since the denominators of the conditional expected payoffs may be zero. 
The authors of \cite{PortakalWindisch} propose a different, slightly stricter definition of dependency equilibria in the boundary cases. This different definition does for example not have the property that all Nash equilibria are dependency equilibria, which, for the definition of Spohn is shown for $2 \times 2$ games in Proposition~\ref{prop: NE are DE}.
For the most part of this paper, we will focus on {\emph{totally mixed dependency equilibria}}, which live in the open simplex $\Delta$, i.e.\ for which the joint probabilities $p_{j_1 \cdots j_n}$ are neither $0$ nor $1$. The totally mixed dependency equilibrium can therefore simply be described by the equations 
 $\mathbb E^{(i)}_k (p) = \mathbb E^{(i)}_{k'} (p) $
for all players $i$ and all pure strategies $k,k'$ of player $i$. 
By multiplying these equations by the denominators, one finds that the totally mixed dependency equilibria can be described via a determinantal variety. 

\begin{definition}
    The \emph{Spohn variety} $\V$ of a game $X$ is defined as the vanishing set of the $2 \times 2$ minors of the matrices $M_1, \ldots, M_n$, given by
    $$M_i(p) := 
    \begin{bmatrix}
    \vdots & \vdots \\
    p_{+\ldots+k+ \ldots +}& \sum\limits_{j_1=1}^{d_1} \cdots \widehat{\sum\limits_{j_i=1}^{d_i}} \cdots \sum\limits_{j_n=1}^{d_n} X^{(i)}_{j_1\cdots k \cdots j_n} p_{j_1\cdots k \cdots j_n} \\
    \vdots & \vdots
    \end{bmatrix}
    \in \mathbb{R}^{d_i \times 2}.
    $$
\end{definition}

In this paper, we will focus mainly on $2 \times 2$ games. For simplicity, we will denote the $2 \times 2$ payoff tables $X^{(1)}$ as $A$ and $X^{(2)}$ as $B$. Then, the conditional expected payoffs are given by

\begin{align*}
&\mathbb E_1^{(1)} (p) =  \frac{a_{11} p_{11} + a_{12} p_{12}}{p_{11}+p_{12}}, & \hspace{-2.7cm}
\mathbb E_2^{(1)} (p) =  \frac{a_{21} p_{21} + a_{22} p_{22}}{p_{21}+p_{22}}, \\
&\mathbb E_1^{(2)} (p) =  \frac{b_{11} p_{11} + b_{21} p_{21}}{p_{11}+p_{21}}, & \hspace{-2.5cm}
\mathbb E_2^{(2)} (p) =  \frac{b_{12} p_{12} + b_{22} p_{22}}{p_{12}+p_{22}}.
\end{align*}
If $p \in \Delta $ (or more general if $p \in \overline{\Delta}$ and  $p_{1+},p_{2+},p_{+1},p_{+2} \neq 0$), then $p$ is a dependency equilibrium if and only if 
$$p \in \V = \mathbb V(\det M_1,\det M_2)\subset \mathbb P_{\mathbb C}^3$$
for
$$
M_1 = \begin{bmatrix}
p_{11}+p_{12} & a_{11}p_{11}+ a_{12}p_{12} \\
p_{21}+p_{22} & a_{21}p_{21}+a_{22}p_{22} 
\end{bmatrix}, \hspace{0.3cm} 
M_2 = \begin{bmatrix}
p_{11}+p_{21} & b_{11}p_{11}+b_{21}p_{21} \\
p_{12}+p_{22} & b_{12}p_{12}+b_{22}p_{22} 
\end{bmatrix} .
$$
More precisely,
\begin{equation}
\label{eq:quadricexample}
\begin{split}
 \det M_1 
&= (a_{21}-a_{11})p_{11}p_{21} + (a_{22}-a_{11})p_{11}p_{22} + (a_{21}-a_{12})p_{12}p_{21} + (a_{22}-a_{12})p_{12}p_{22},\\
 \det M_2 
&= (b_{12}-b_{11}) p_{11}p_{12} + (b_{22}-b_{11})p_{11}p_{22} + (b_{12}-b_{21})p_{12}p_{21} + (b_{22}-b_{21})p_{21}p_{22}.
\end{split}
\end{equation}

\begin{example}[Prisoner's Dilemma]\label{PrisonersDilemma}
    Consider the $2 \times 2$ game with payoff tables
    $$
    A =
    \begin{pmatrix}
    2& 0 \\ 3 & 1
    \end{pmatrix}, \hspace{0.3cm}
    B=
    \begin{pmatrix}
    2 & 3 \\ 0 & 1
    \end{pmatrix}.
    $$
    The $2 \times 2$ games are commonly represented in bimatrix form, where each cell shows the payoffs for both players: the first value is the payoff for player 1, and the second value is for player 2.
    \[
\begin{array}{c|cc}
    & \text{Cooperate} & \text{Defect} \\
    \hline
    \text{Cooperate} & (2,\,2) & (0,\,3) \\
    \text{Defect}    & (3,\,0) & (1,\,1) \\
\end{array}
\]
    The Spohn variety is defined by the determinants of the matrices  namely:

    \begin{align*}
    M_1 = \begin{bmatrix}
    p_{11}+p_{12} & 2p_{11} \\
    p_{21}+p_{22} & 3p_{21}+p_{22} 
    \end{bmatrix}, \hspace{0.3cm}
    M_2 = \begin{bmatrix}
    p_{11}+p_{21} & 2p_{11} \\
    p_{12}+p_{22} & 3p_{12}+p_{22} 
    \end{bmatrix},
    \end{align*}

    \begin{align*}
    \V =
    \mathbb V(
     p_{11}p_{21} - p_{11}p_{22} + 3 p_{12}p_{21} + p_{12}p_{22},
    p_{11}p_{12} -p_{11}p_{22} +3 p_{12}p_{21} +p_{21}p_{22}).
    \end{align*}
\end{example}

The Spohn variety for generic games has been studied in detail in \cite[§4]{PortakalSturmfels}. We state some of the key results here. In particular, Theorem~\ref{parametrization} explains why the case of $2 \times 2$ games is unique and needs to be studied separately.

\begin{theorem}\cite[Theorem 6]{PortakalSturmfels}\label{properties}
If the payoff tables of a game are generic, then the Spohn variety $\V$ is irreducible of codimension $d_1 + \ldots + d_n -n$ and degree $d_1 \cdots d_n$. The intersection of $\V$ with the Segre variety in the open simplex $\Delta$ is precisely the set of totally mixed Nash equilibria for the game $X$.
\end{theorem}

\begin{theorem}\cite[Theorem 8]{PortakalSturmfels}\label{parametrization}
If $n=d_1=d_2=2$ then the Spohn variety $\V$ is an elliptic curve. In all other cases, the Spohn variety is rational, represented by a map onto $(\mathbb P^1)^n$ with linear fibers.
\end{theorem}

It is important to bear in mind that Theorems~\ref{properties} and \ref{parametrization} above are established for generic games and do not necessarily apply to specific games. For example, in the Prisoner's Dilemma from Example~\ref{PrisonersDilemma}, the Spohn curve is both reducible and singular. In addition to Theorem~\ref{properties}, it is natural to ask whether Nash equilibria on the boundary of the simplex are also dependency equilibria according to Definition~\ref{DESpohn}. Furthermore, one may investigate dependency equilibria that Pareto dominate Nash equilibria. The following result appears in \cite{Spohn2} without a formal proof. 
\begin{proposition}\label{prop: NE are DE}
A Nash equilibrium is a dependency equilibrium for any $2\times 2$-game.

\begin{proof}
Nash equilibria can be computed following \cite[Section 6]{NE}.
Let $p_k^{(i)}$ be the probability of player $i$ choosing $k \in [2]$. The set of Nash equilibria of $2 \times 2$-games are thus contained in $\overline{\Delta_1} \times \overline{\Delta_1}$. For a totally mixed Nash equilibrium, the result follows from Theorem~\ref{properties}.
Without loss of generality consider a Nash equilibrium $((0,1),(p^{(2)}_1, p^{(2)}_2)) \in \overline{\Delta_1} \times \overline{\Delta_1}$ and identify it with its image under the Segre embedding $p:=(p_{11}, p_{12}, p_{21}, p_{22})=(0,0,p^{(2)}_1, p^{(2)}_2) \in \overline{\Delta} = \Delta_3$.

Let $p^{(2)}_1, p^{(2)}_2 \neq 0 $. Then, since the payoff is linear in each distribution and usually takes its optimal value at $p^{(2)} \in \{ (1,0), (0,1) \}$,  for $p$ to be a Nash equilibrium, it must be $a_{21}p^{(2)}_1+a_{22}p^{(2)}_2 \geq a_{11}p^{(2)}_1 + a_{12}p^{(2)}_2$ and $b_{21}=b_{22}$.
Assume $a_{11} \leq a_{12}$. We can write $a_{12}= a_{11}+s$. Define a sequence 
$(p^{(r)})_{r \in \mathbb N_{\geq m}}$ in $\overline{\Delta}$ by 
$p^{(r)} = (\frac{1}{r}, \frac{1}{r^2}, p^{(2)}_1 - \frac{1}{r}, p^{(2)}_2- \frac{1}{r^2})$, 
where $m$  is chosen sufficiently large for this to be a distribution. 
Then $p^{(r)} \stackrel{r \to \infty}{\longrightarrow} p$ and
\begin{align*}
\lim\limits_{r \to \infty} \frac{b_{11} p_{11}^{(r)} + b_{21} p_{21}^{(r)}}{p_{11}^{(r)}+p_{21}^{(r)}} = b_{21} = b_{22} =
\lim\limits_{r \to \infty}  \frac{b_{12} p_{12}^{(r)} + b_{22} p_{22}^{(r)}}{p_{12}^{(r)}+p_{22}^{(r)}},
\end{align*}
\begin{align*}
\lim\limits_{r \to \infty} \frac{a_{11} p^{(r)}_{11} + a_{12} p^{(r)}_{12}}{p^{(r)}_{11}+p^{(r)}_{12}}
= a_{11} + \lim\limits_{r \to \infty} \frac{s p_{12}^{(r)}}{p_{11}^{(r)}+ p_{12}^{(r)}} = a_{11} + \lim\limits_{r \to \infty} \frac{s}{r+1} = a_{11} \\
\leq a_{11}p^{(2)}_1 + a_{12}p^{(2)}_2
\leq a_{21} p_1^{(2)} + a_{22} p_2^{(2)} =
\lim\limits_{r \to \infty} \frac{a_{21} p^{(r)}_{21} + a_{22} p^{(r)}_{22}}{p^{(r)}_{21}+p^{(r)}_{22}}.
\end{align*}
If $a_{12} \leq a_{11}$ define the sequence $(p^{(r)})_{r \in \mathbb N_{\geq m}}$ by $p^{(r)} =  (\frac{1}{r^2}, \frac{1}{r}, p^{(2)}_1 - \frac{1}{r}, p^{(2)}_2- \frac{1}{r^2})$ and proceed similarly.

Now assume without loss of generality $p^{(2)} = (0,1)$. Then for $p$ to be a Nash equilibrium, it must be $a_{12} \leq a_{22}$ and $b_{21} \leq b_{22}$.
Depending on the other entries of the payoff matrices, we define a sequence $(p^{(r)})_{r \in \mathbb N_{\geq m}}$. In each case, we take $r$ starting at a sufficiently large integer for $p^{(r)}$ to be a distribution. 
\begin{itemize}
\item $a_{11} \leq a_{12}$ and $b_{11} \leq b_{21}$: $p^{(r)} = (\frac{1}{r}, \frac{1}{r^2}, \frac{1}{r^2} , 1- \frac{1}{r}- \frac{2}{r^2})$.
\item $a_{11} \geq a_{12}$ and $b_{11} \geq b_{21}$: $p^{(r)} = (\frac{1}{r^2}, \frac{1}{r}, \frac{1}{r} , 1- \frac{2}{r}- \frac{1}{r^2}) $.
\item $a_{11} \leq a_{12}$ and $b_{11} \geq b_{21}$: $p^{(r)} = (\frac{1}{r^2}, \frac{1}{r^3}, \frac{1}{r} , 1- \frac{1}{r}- \frac{1}{r^2} - \frac{1}{r^3}) $.
\item $a_{11} \geq a_{12}$ and $b_{11} \leq b_{21}$: $p^{(r)} = (\frac{1}{r^2}, \frac{1}{r}, \frac{1}{r^3} , 1- \frac{1}{r}- \frac{1}{r^2} - \frac{1}{r^3}) $.
\end{itemize}

\noindent Then, $p^{(r)} \stackrel{r \to \infty}{\longrightarrow} p$ and one obtains the desired inequalities:

\begin{align*}
\lim\limits_{r \to \infty} \frac{a_{11} p^{(r)}_{11} + a_{12} p^{(r)}_{12}}{p^{(r)}_{11}+p^{(r)}_{12}}
= \min \{a_{11}, a_{12} \} \leq a_{22} =
\lim\limits_{r \to \infty} \frac{a_{21} p^{(r)}_{21} + a_{22} p^{(r)}_{22}}{p^{(r)}_{21}+p^{(r)}_{22}}, \end{align*}
\begin{align*}
\lim\limits_{r \to \infty} \frac{b_{11} p_{11} + b_{21} p_{21}}{p_{11}+p_{21}} = \min \{b_{11},b_{21} \} \leq b_{22} =
\lim\limits_{r \to \infty}  \frac{b_{12} p_{12} + b_{22} p_{22}}{p_{12}+p_{22}}.
\end{align*}

\end{proof}
\end{proposition}

\begin{proposition}\label{prop: prisoners dilemma type}
Let $X$ be a Prisoner's Dilemma-type $2 \times 2$ game, represented in bimatrix form as follows:
\[
\begin{array}{c|cc}
    & \text{Cooperate} & \text{Defect} \\
    \hline
    \text{Cooperate} & (a_{11},\, a_{11}) & (a_{12},\, a_{21}) \\
    \text{Defect}    & (a_{21},\, a_{12}) & (a_{22},\, a_{22}) \\
\end{array}
\]
where the payoffs satisfy $a_{21} > a_{11} > a_{22} > a_{12}$. Then, there exist infinitely many dependency equilibria that Pareto dominate the unique Nash equilibrium $(0,0,0,1) \in \overline{\Delta}$ of $X$. In particular, the point $(1, 0, 0, 0) \in \overline{\Delta}$, corresponding to both players choosing to cooperate, is a dependency equilibrium.
\end{proposition}

\begin{proof}
By \cite[Theorem 19]{PortakalSturmfels}, the determinant of the Konstanz matrix $K_X(\Pi)$ of the game $X$ is the image of the Spohn curve $\mathcal{V}_X$ under the map $\pi_X$ defined by the expected payoffs of each player $(\pi_1, \pi_2)$ which is as follow:
\[
K_X(\Pi) = 
\begin{bmatrix}
\pi_1 - a_{11} & \pi_1 - a_{12} & 0 & 0 \\
0 & 0 & \pi_1 - a_{21} & \pi_1 - a_{22} \\
\pi_2 - a_{11} & 0 & \pi_2 - a_{12} & 0 \\
0 & \pi_2 - a_{21} & 0 & \pi_2 - a_{22}
\end{bmatrix}
\]
The ideal defined by $\text{det}(K_X(\Pi))$ contains the component $\pi_1 - \pi_2$ that passes through $(a_{11},a_{11})$ and $(a_{22},a_{22})$ in $\pi_X(\overline{\Delta})$. The intersection of the kernel of $K_X(\Pi)$ for this component with the open simplex $\Delta$ is not empty and positive-dimensional. Thus, the first statement follows. For the second statement, for \( r \geq 2 \), let
\[
p^{(r)} = \left(1 - \frac{1}{r} - \frac{\lambda}{r^2} - \frac{1-\lambda}{r^2},\; \frac{1}{r},\; \frac{\lambda}{r^2},\; \frac{1-\lambda}{r^2} \right)
\]
where
\[
\lambda = \frac{a_{11} - a_{22}}{a_{21} - a_{22}} \in (0,1).
\]
This sequence converges to \( (1,0,0,0) \) and satisfies the following equality and inequality:
\begin{align*}
\lim_{r \to \infty} \mathbb{E}_1^{(1)}(p^{(r)}) &= a_{11} = \lambda a_{21} + (1-\lambda)a_{22} = \lim_{r \to \infty} \mathbb{E}_2^{(1)}(p^{(r)}), \\[0.5em]
\lim_{r \to \infty} \mathbb{E}_1^{(2)}(p^{(r)}) &= a_{11} > a_{22} = 
\lim_{r \to \infty} \mathbb{E}_2^{(2)}(p^{(r)}).
\end{align*}
\end{proof}
\begin{remark}\label{rem: Pareto dom}
While generically, Nash equilibria consist of finitely many points, the Spohn variety clearly does not. It is natural to ask if this means that by looking at the dependency equilibria, one can always find a better outcome than the one coming from Nash equilibria. The \emph{payoff curve}, as defined in \cite{PortakalSturmfels}, describes the payoffs coming from totally mixed dependency equilibria. It can be computed via the determinant of the Konstanz matrix, the matrix that gives the parametrization in Theorem~\ref{parametrization}.  For any two points on the curve, if one has higher payoffs for both players, it Pareto dominates the other one. 
For generic $2 \times 2$ games, there is only one totally mixed Nash equilibrium, which can be computed as in \cite[Theorem 6.6]{NE}. Using quantifier elimination in \texttt{Mathematica} \cite{Mathematica}, we propose \href{https://mathrepo.mis.mpg.de/elliptic_curves_game_theory/index.html#computations-for-section-2}{{\color{magenta} a method}} in \cite{mathrepo} to check whether this Nash point is Pareto dominated by a dependency equilibrium.
\end{remark}

\section{Denseness of real points}
\label{sec:denseness}
When we are interested in dependency equilibria, we are automatically interested in the Spohn variety. This is, because, as seen above, dependency equilibria for which the denominators of the conditional expected payoff do not vanish can be completely described as the intersection of the Spohn variety with the open probability simplex $\Delta$. From an algebraic viewpoint, it is much nicer to work with a variety, especially one we already know basic properties about, as seen above. The discrepancy between studying the Spohn variety and actual dependency equilibria arises not only from the boundary cases but also from the fact that we consider the Spohn variety as a projective variety over the complex numbers, while the probabilities making up the dependency equilibria of course have to be real. We can easily bridge this gap, and for example adopt statements of dimension and degree to the real part, if we can show that the real points lie dense within the Spohn variety.
In general, this can be achieved via the parametrization in Theorem~\ref{parametrization}. We need to focus on the case of $2 \times 2$ games separately, in which the Spohn variety is a curve in $\mathbb P^3$ given by two quadrics.

\subsection{Reducibility of the Spohn cubic}
\label{sec:redspohn}
For a $2 \times 2$ game with payoff matrices $X^{(1)}=(a_{ij})$, $X^{(2)}=(b_{ij})$, the Spohn variety is defined by the two quadrics $\det(M_1)$ and $\det(M_2)$. We may eliminate $p_{2,2}$ in these equations while the elliptic curve remains the same up to isomorphy.
After a relabeling of the variables to $x=p_{11}, y=p_{12}, z=p_{21}$ for simplicity, the resulting planar model is the ternary cubic $\mathcal{C} \subset \mathbb P^2$, given by
\begin{equation}\label{eq: Spohn cubic}
f = c_1x^2y +c_2x^2z + c_3xy^2 + c_4xz^2 + c_5y^2z + c_6yz^2 + c_7xyz,
\end{equation}
where
\begin{align*}
c_1 &= (a_{11}-a_{22})(b_{11}-b_{12}), \\
c_2 &= (a_{11}-a_{21})(b_{22}-b_{11}), \\
c_3 &= (a_{12}-a_{22})(b_{11}-b_{12}), \\
c_4 &= (a_{11}-a_{21})(b_{22}-b_{21}), \\
c_5 &= (a_{12}-a_{22})(b_{21}-b_{12}), \\
c_6 &= (a_{12}-a_{21})(b_{22}-b_{21}), \\
c_7 &= (a_{12}-a_{21})(b_{22}-b_{11}) + (a_{11}-a_{22})(b_{21}-b_{12}).
\end{align*}
A ternary cubic of this form is called a \emph{Spohn cubic}. For reasons that will become clearer later on, we are interested in the irreducible components of the Spohn cubic. According to Theorem~\ref{properties}, for generic payoff matrices, the Spohn variety, and therefore also the Spohn cubic $\mathcal C$, is irreducible. But what conditions must the entries of the payoff tables fulfill in order for $\mathcal C$ to be reducible?

\begin{remark}\label{rem: cubic is zero}
    The cubic equation $f$ is zero and therefore $\mathcal C = \mathbb P^2$ if and only if one of the following holds
    \begin{enumerate}
        \item One of the payoff tables is constant.
        \item $a_{11}=a_{21}$, $a_{12}=a_{22}$, $b_{11}=b_{12}$, $b_{21}=b_{22}$
        \item $a_{11}=a_{12}=a_{22}$, $b_{11}=b_{21}=b_{22}$
        \item $a_{11}=a_{12}=a_{21}$, $b_{11}=b_{12}=b_{21}$
    \end{enumerate}
The second case is exactly when $\det M_1 = \det M_2$.    
\end{remark}
\noindent The following result answers, to some extent, the questions posed in \cite[Problem 4.3]{PortakalWindisch}. 

\begin{theorem}
\label{thm: cases}
The cubic $\mathcal C$ is reducible if and only if $f$ is non-zero and one of the following cases holds: \\

\begin{minipage}{0.3\textwidth}
    \begin{enumerate}[leftmargin = 0.95cm, itemsep = 0.6ex]
        \item $a_{11}=a_{12}$ 
        \item $a_{11}=a_{21}$ 
        \item $a_{21}=a_{22}$ 
    \end{enumerate}
\end{minipage}
\begin{minipage}{0.22\textwidth}
    \begin{enumerate}[leftmargin = 0cm, itemsep = 0.6ex]        
        \setcounter{enumi}{3}        
        \item $b_{11}=b_{12}$ 
        \item $b_{11}=b_{21}$ 
        \item $b_{12}=b_{22}$  
    \end{enumerate}
\end{minipage}
\begin{minipage}[b]{0.3\textwidth}
    \begin{enumerate}[leftmargin = 0cm, itemsep = 0.6ex]
        \setcounter{enumi}{6}
        \item $a_{12}=a_{22}, b_{21}=b_{22}$
        \item $a_{12}=a_{21}, b_{12}=b_{21}$
    \end{enumerate}
\end{minipage} \\
\vspace{0.15cm}
\begin{enumerate}[itemsep = 1.7ex]
\setcounter{enumi}{8}
\item
            $ \hspace{0.14cm} 0= a_{12}(b_{12}-b_{22}) + a_{21}(b_{22}-b_{21}) + a_{22} (b_{21}-b_{12}) $ \vspace{0.1cm} \\
             $\hspace*{0.46cm} = a_{11}(b_{22}-b_{12}) + a_{21}(b_{11}-b_{22}) + a_{22}(b_{12}-b_{11}) $ \vspace{0.1cm}\\
            $ \hspace*{0.46cm} =  a_{11}(b_{22}-b_{21}) + a_{12}(b_{11}-b_{22}) + a_{22}(b_{21}-b_{11}) $
\item 
           \hspace{0.14cm}$ 0 = a_{11}(b_{12}-b_{21}) + a_{12}(b_{21}-b_{22}) + a_{21} (b_{22}-b_{12}) $ \vspace{0.1cm}\\            
           \hspace*{0.46cm}$= a_{12}(b_{11}-b_{21}) + a_{21}(b_{12}-b_{11}) + a_{22}(b_{21}-b_{12}) $ \vspace{0.1cm}\\
           \hspace*{0.46cm}$ =  a_{11}(b_{11}-b_{21}) + a_{21}(b_{22}-b_{11}) + a_{22}(b_{21}-b_{22}) $ 

\item 
           \hspace{0.14cm}$0 = a_{12}(b_{22}-b_{21}) + a_{21}(b_{12}-b_{22}) + a_{22} (b_{21}-b_{12}) $ \vspace{0.1cm}\\ 
           \hspace*{0.46cm}$= a_{11}(b_{22}-b_{21}) + a_{21}(b_{11}-b_{22}) + a_{22}(b_{21}-b_{11}) $ \vspace{0.1cm}\\
           \hspace*{0.46cm}$ =  a_{11}(b_{22}-b_{12}) + a_{12}(b_{11}-b_{22}) + a_{22}(b_{12}-b_{11}) $

\item 
        \hspace{0.14cm}$0= a_{11}(b_{12}-b_{21}) + a_{12}(b_{22}-b_{12}) + a_{21} (b_{21}-b_{22}) $ \vspace{0.1cm}\\ 
           \hspace*{0.46cm}$= a_{12}(b_{11}-b_{12}) + a_{21}(b_{21}-b_{11}) + a_{22}(b_{12}-b_{21}) $ \vspace{0.1cm}\\
           \hspace*{0.46cm}$ =  a_{11}(b_{11}-b_{21}) + a_{12}(b_{22}-b_{12}) + a_{21}(b_{21}-b_{11}) + a_{22}(b_{12}-b_{22})$       
\end{enumerate}

\begin{proof}
The \href{https://mathrepo.mis.mpg.de/elliptic_curves_game_theory/#proof-of-lemma-3-2}{{\color{magenta}detailed computations}} can be found in \cite{mathrepo}.
The ternary cubic $\mathcal C \subset \mathbb P^2$ is reducible if and only if there exists a projective line completely contained in it.
 Take three lines $L_1$, $L_2$, $L_3$ which do not have a common point of intersection and consider their intersections $X_i = L_i \cap \mathcal C$ with the cubic. Clearly, if any of these intersections is the whole line, then $\mathcal C$ is reducible. If this is not the case, then, since any two lines in $\mathbb{P}^2$ intersect, any line contained in $\mathcal C$ passes through a pair of distinct points $(p_1,p_2) \in X_1 \times (X_2 \cup X_3)$. \\

\noindent
Consider the projective lines $L_1= \mathbb V(x)$, $L_2 = \mathbb V(y)$ and $L_3 = \mathbb V(z)$. We have
\begin{align*}
f(0,y,z) = yz(c_5 y + c_6 z), \\
f(x,0,z) = xz(c_2 x + c_4 z), \\ 
f(x,y,0) = xy(c_1 x + c_3 y).
\end{align*}
If any of these is zero, then this means that the entire corresponding line is contained in the cubic. Hence, if $c_5=c_6=0$, $c_2=c_4=0$ or $c_1=c_3=0$, the cubic is reducible and we are done. If not, we look at the points of intersection $X_i = \mathcal C \cap L_i$ of the cubic with these lines, which can be obtained through the zeros of the polynomials above. Namely,

\begin{align*}
X_1 = \{[0:0:1],[0:1:0],[0:1:-\frac{c_5}{c_6}]\}, \\ 
X_2 = \{[0:0:1],[1:0:0],[1:0:-\frac{c_2}{c_4}]\}, \\ 
X_3 = \{[0:1:0],[1:0:0],[1:-\frac{c_1}{c_3}:0]\},
\end{align*}
where the last element of each set is only contained if the corresponding denominator is nonzero.\\

\noindent
We now take pairs of distinct points $(p_1,p_2) \in X_1 \times ((X_2 \setminus X_1) \cup (X_3 \setminus X_1))$.  
By inserting both points into a line equation $ax+by+cz=0$ we obtain the defining equation of the line going through them. For example, the first pair gives us $c=0$ and $a-\frac{c_1}{c_3}b=0$, which results in the given line.
If one of these lines is contained in $\mathcal C$, then, under the given assumptions, $\mathcal C$ is reducible; if not, it must be irreducible.
$$ 
\begin{array}{l|l}
 (p_1,p_2) & \text{Line through } p_1 \text{ and } p_2  \\
\hline
[0:0:1],[1: - \frac{c_1}{c_3} :0] & V( \frac{c_1}{c_3} x + y) \\
  {[0:1:0]},[1:0: - \frac{c_2}{c_4}] & V( \frac{c_2}{c_4} x + z) \\
{[0:1: - \frac{c_5}{c_6} ]},[1:0:0] & V(\frac{c_5}{c_6} y + z) \\
 {[0:1: - \frac{c_5}{c_6} ]},[1:0: - \frac{c_2}{c_4} ] & V( \frac{c_2}{c_4}x+ \frac{c_5}{c_6} y +z )
\\
{[0:1: - \frac{c_5}{c_6} ]},[1: - \frac{c_1}{c_3} :0] & V( \frac{c_1 c_5}{c_3 c_6}x+ \frac{c_5}{c_6} y +z )
\end{array}
$$
Notice that the pairs only actually appear here if all denominators of coordinates are nonzero.\\

\noindent
Inserting these lines into the conic equation and decomposing yields 
\begin{align*}
 f(x, - \frac{c_1}{c_3} x, z) 
    & = \frac{- xz (a_{21}-a_{22})(a_{11}-a_{12})((b_{11}-b_{22})x+(b_{21}-b_{22})z)}{a_{12}-a_{22}}, \\
 f(x,y,- \frac{c_2}{c_4} x) 
    & = \frac{xy (b_{12}-b_{22})(b_{11}-b_{21})((a_{11}-a_{22})x + (a_{12}-a_{22})y )}{b_{21}-b_{22}}, \\
 f(x,y,- \frac{c_5}{c_6} y) 
    & = \frac{xy(d_3 x + e_3 y)}{(a_{12}-a_{21})^2(b_{21}-b_{22})}, \\
 f(x,y,- \frac{c_2}{c_4} x - \frac{c_5}{c_6} y) 
    & = \frac{-xy (d_4 x + e_4 y)}{(a_{12}-a_{21})^2(b_{21}-b_{22})}, \\  
 f(x,y,- \frac{c_1 c_5}{c_3 c_6} x - \frac{c_5}{c_6} y)  
    & = \frac{x ((a_{11}-a_{22})x + (a_{12}-a_{22})y) (d_5 x + e_5 y)}{(a_{12}-a_{21})^2(b_{21}-b_{22})},
\end{align*}
where the $d_l$ and $e_l$ are very long polynomials in the entries of the payoff matrices with integer coefficients. Note that, under the given assumptions, the denominators are nonzero.
The cubic is now reducible if, under the given assumptions, one of the numerators of these polynomials is zero everywhere or, more precisely, if one of the factors in their decompositions above is zero. \\
 
\noindent 
First, let us take a look at the simpler factors.
The first polynomial, $f(x, - \frac{c_1}{c_3} x, z)$, is zero if and only if $a_{21}=a_{22}$, $a_{11}=a_{12}$ or $b_{11}=b_{21}=b_{22}$. We assume here that $c_3 \neq 0$, but actually reducibility follows from these cases even if this is not satisfied: Indeed, if $c_3=0$, then there are two possibilities. If $b_{11}=b_{12}$, then $c_1=c_3=0$.
If $a_{12}=a_{22}$, then 
$a_{12}=a_{21}=a_{22}$ implies $c_5=c_6=0$ and 
$a_{11}=a_{12}=a_{22}$ implies $c_1=c_3=0$. The case that $b_{11}=b_{21}=b_{22}$ implies $c_2=c_4=0$ anyways and can therefore actually be omitted. We are left with the cases (1) and (3) from the classification.\\
\noindent
Very similarly, the second polynomial gives us the cases (5) and (6). The case $a_{11}=a_{12}=a_{22}$ coming from the last factor cannot occur, since we assume $c_3 \neq 0$ here.\\

\noindent
We just used that $\mathcal C$ is  reducible if $c_1=c_3=0$, $c_2=c_4=0$ or $c_5=c_6=0$. Given the simpler cases we just obtained, this now reduces to $b_{11}=b_{12}$, $a_{11}=a_{21}$, $(a_{12}=a_{22} \wedge b_{22}=b_{21})$ or $(a_{12}-a_{21} \wedge b_{21}-b_{12})$, namely the cases (2), (4), (7) and (8). \\

\noindent
For the cases under which the long factors vanish we consider the ideals $J_i = (d_i,e_i)$, since we are interested in the conditions on the payoff matrix entries that guarantee $d_i=e_i=0$. \\

\noindent
Let us decompose for example $J_5$: 
\begin{align*}
J_5 = & \hspace{4.5mm} (b_{21}-b_{22}, b_{12}-b_{22}) 
\cap (b_{12}-b_{22}, a_{11}-a_{21}) 
\cap (b_{12}-b_{21}, b_{11}-b_{21}) \\ &
\cap (b_{12}-b_{21}, a_{12}-a_{21})
\cap (b_{11}-b_{21}, a_{11}-a_{21})
\cap (a_{12}-a_{21}, a_{11}-a_{21}) \\ &
\cap I_{11} \cap I_{12}
\end{align*}
Here, $I_{11}$ and $I_{12}$ are the ideals corresponding to the cases (11) and (12). One can check that the vanishing of all the other components is already covered by the cases (1) to (8). In order to actually obtain (11) and (12) from $I_{11}$ and $I_{12}$ we need to show that the assumption that $c_3 \neq 0$ and $c_6 \neq 0$ is not necessary. Let for instance $c_3=0$. If $b_{11}=b_{12}$, then that is case (4). To see what happens if $a_{12}=a_{22}$ we look at the decomposition of the two ideals
\begin{align*}
 & \hspace{1mm} I_{11} + (a_{12}-a_{22})    \\
= & \hspace{1mm} (b_{12}-b_{22}, a_{12}-a_{22}, a_{21}b_{11} - a_{22}b_{11}-a_{11}b_{21} + a_{22}b_{21} + a_{11}b_{22} - a_{21}b_{22}) \\
&\cap (a_{21}-a_{22},a_{12}-a_{22}, a_{11}-a_{22}),
\end{align*}
\begin{align*}
 & \hspace{1mm} I_{12} + (a_{12}-a_{22})  \\
= & \hspace{1mm} (b_{11}-b_{21},a_{12}-a_{22}, a_{11}b_{12} - a_{22}b_{12}-a_{11}b_{21}+a_{21}b_{21}-a_{21}b_{22}+a_{22}b_{22}) \\
& \hspace{1mm} \cap  (a_{21}-a_{22},a_{12}-a_{22}, a_{11}-a_{22}).
\end{align*}
The second components cannot vanish, since $f$ is nonzero (Remark~\ref{rem: cubic is zero}). The first components vanishing is already covered by the existing cases. Hence, it is not necessary to assume that $c_3 \neq 0$ and, similarly, neither is $c_6 \neq 0$. In this case we also note that $f(x,y,- \frac{c_1 c_5}{c_3 c_6} x - \frac{c_5}{c_6} y)$ also vanishes if $a_{11}=a_{12}=a_{22}$, but this is already covered by case (1). \\

\noindent
For $J_3$ and $J_4$, we can proceed similarly. Decomposing $J_3$ gives us the cases (9) and (10), and decomposing $J_4$ gives us the cases (11) and (12) again. Assumptions are required for none of both.
\end{proof}
\end{theorem}

Given the cases from Theorem~\ref{thm: cases} and assuming that the payoff tables are otherwise generic, one can compute the actual \href{https://mathrepo.mis.mpg.de/elliptic_curves_game_theory/#decomposition-of-spohn-cubic}{\textcolor{magenta}{irreducible components of the Spohn cubic}}, as is done in \cite{mathrepo}. We see that here, the Spohn cubic always decomposes into a line and a conic. 

Let us take a closer look at what these cases look like.
For the first 8 cases, it is quite clear that the Spohn cubic is reducible. The longer cases are not as obvious but there is still some pattern to them. In fact, the game Prisoner's Dilemma from Example~\ref{PrisonersDilemma} falls under case (9). Notice that the equalities in the cases are invariant under operations on the payoff tables that preserve the Spohn curve (see also Example \ref{Examplejinv}).

\begin{remark}\label{rem: combinatorialstructure}
Denote by $I_9$, $I_{10}$, $I_{11}$, $I_{12}$ the ideals spanned by the (quadric) polynomials in the cases (9) to (12) of Theorem~\ref{thm: cases}. 
First we notice \href{https://mathrepo.mis.mpg.de/elliptic_curves_game_theory/#remark-3-3}{\textcolor{magenta}{computationally}} (\cite{mathrepo}), that 
\begin{equation}\label{eq: combinatorial quadrics}
\begin{split}
I_9 &=  I_9 + (a_{11}(b_{12}-b_{21}) + a_{12}(b_{11}-b_{12}) + a_{21}(b_{21}-b_{11})) \\
I_{10} &= I_{10} + (a_{11}(b_{11}-b_{12})+ a_{12}(b_{22}-b_{11}) + a_{22}(b_{12}-b_{22})) \\
I_{11} &= I_{11} + (a_{11}(b_{12}-b_{21}) + a_{12}(b_{21}-b_{11}) + a_{21}(b_{11}-b_{12}))\\
I_{12} &=  I_{12} + (a_{11}(b_{11}-b_{21}) + a_{12} (b_{22}-b_{11}) + a_{22} (b_{21}-b_{22}), \\
& \hspace{1.75cm} a_{11}(b_{11}-b_{12}) + a_{21}(b_{22}-b_{11}) + a_{22}(b_{12}-b_{22})).
\end{split}
\end{equation}
The original generators of the ideals $I_9,I_{10},I_{11}$ (i.e., the polynomials in Theorem~\ref{thm: cases}) are interchangeable with these new polynomials that have been added in (\ref{eq: combinatorial quadrics}): any choice of three of the four generators (i.e.\ of the three original generators and the generator added in (\ref{eq: combinatorial quadrics})) is sufficient to generate the entire ideal $I_m$, for $m=9,10,11$. 
For case (12), the third generator of $I_{12}$ does not fit in with the others. Here, any choice of three of the first two original generators of $I_{12}$ and the two additional generators in (\ref{eq: combinatorial quadrics}) is sufficient to generate the entire ideal. We can therefore forget about the long term in (12) and replace it with the shorter ones that fit in with the other cases.

\noindent
For each of the generating four polynomials of $I_m$, there is exactly one $a_{ij}$ that does not occur in the polynomial. These $a_{ij}$ are pairwise different within the four generators. Denote by $f^{(m)}_{ij}$ the polynomial among the generators of $I_m$ that does not contain $a_{ij}$. 
These 4-element sets of polynomials can be represented in the payoff tables 
 $A$ and $B$ as follows: \\[0.2cm]

\begin{figure}[h!]
    \centering
    \input{tikzimg}
    \caption{\small For a polynomial $f^{(m)}_{lk}$, the dots on the left side in $A$ represent the variables $a_{ij}$ appearing in $f^{(m)}_{lk}$. The lines on the right side of the same color in $B$ represent the two variables $b_{ij}$ that occur in the same monomial as $a_{ij}$ within $f^{(m)}_{lk}$. For example, the yellow line in $B$ for $f^{(9)}_{11}$ represents the monomials $a_{12}b_{12}$ and $a_{12}b_{22}$ in $f^{(9)}_{11}$.}\label{fig: combinatorial description}
   
\end{figure}
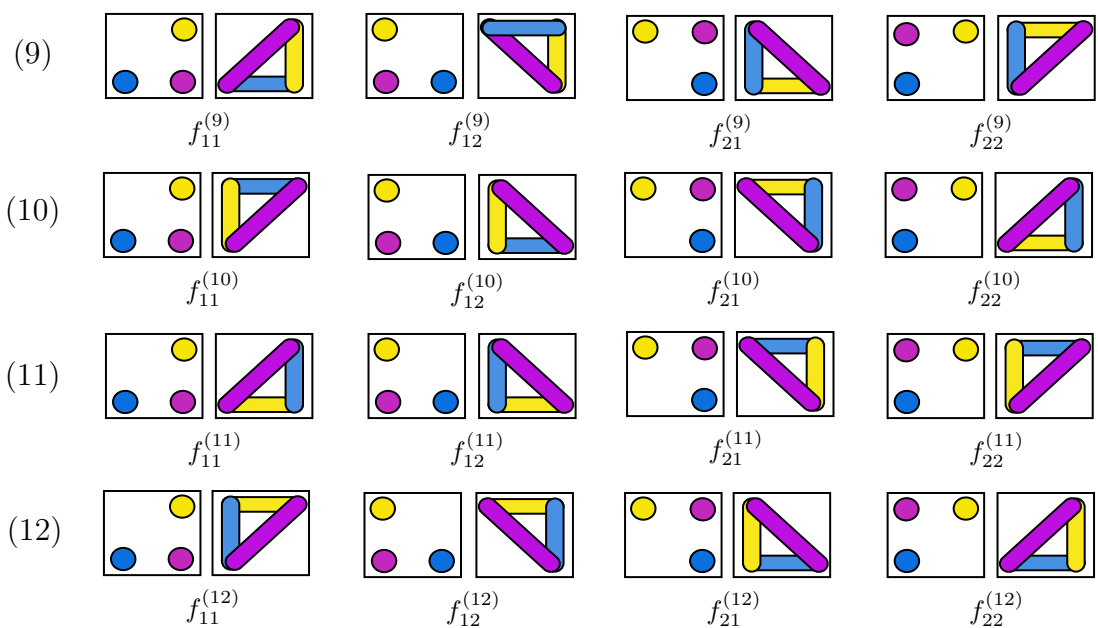

\vspace{0.3cm}

\noindent There is only one way (up to sign) to arrange monomials in $f_{lk}^{(m)}$ such that each variable occurs exactly in two monomials and one monomial that it occurs in has a negative sign and one has a positive sign. Therefore, Figure~\ref{fig: combinatorial description} combinatorially determines the polynomials $f^{(m)}_{lk}$. 

\noindent
The triangles in $B$ either agree in their placement with the triangles in $A$, or they are mirrored on the hypotenuse of $A$. For a fixed $m = 9,10,11,12$, either both triangles in $A$ with the hypotenuse on the diagonal are mirrored on it in $B$, or they both remain the same in $B$. The same holds for the anti-diagonal. We also notice that, for the (purple) point $a_{ij}$ in the triangle in $A$ that is not on the hypotenuse, the corresponding line in $B$ is always the hypotenuse of $B$. For the other (blue and yellow) points $a_{ij}$ in the triangle, either the corresponding edge in $B$ goes through the point $b_{ij}$ in the same position or through the point that lies opposite of $b_{ij}$ in $B$. Within each case (m), for both diagrams in which the triangles have an edge on the diagonal, exactly one of the two choices explained above holds for all blue and yellow points. The same holds for the anti-diagonal. 
\end{remark}

\subsection{Finding real smooth points}

As hinted, knowing about the irreducible components of a variety is useful in determining the denseness of its real part.
A pertinent comment to explicitly state here is that not all quadrics defined over $\mathbb R$ have real points, and not all quadrics with dense real points have a common smooth real point. This is different from the case of irreducible plane cubics with a distinguished point for which a smooth real point always exists.
\begin{theorem}\cite[Theorem 2.2.9]{Mangolte}
Let $V$ be an irreducible complex affine variety defined by real polynomials. If $V$ has a smooth point with real coordinates, then $V(\mathbb R)$, the set of real points of $V$, lies Zariski dense in $V$.
\end{theorem}

By covering the ambient projective space with dense open affine subsets, we may apply this to the Spohn cubic, as well as the Spohn variety.

\begin{lemma}\label{lem: denseness for plane cubic}   
Given that the defining equation $f$ is non-zero,  every irreducible component of the Spohn cubic $\mathcal C$
contains a smooth point with real coordinates.
 In particular, the set of real points $\mathcal C (\mathbb R)$ lies Zariski dense in $\mathcal C$.
\begin{proof}
If $f$ is zero, the real points trivially lie dense in $\mathcal C$. Hence assume that $f \neq 0$. If $f$ is irreducible, then the cubic is an elliptic curve and has real points.  
If the cubic $\mathcal C$ is reducible, it is either the union of a line and a conic or the union of three lines. All lines in $\mathbb P^2$ are smooth and contain real points. While all irreducible conics are projectively equivalent to $x^2 +yz$ and therefore smooth, it can happen that they contain no real smooth points. Hence, it suffices to take a closer look at those conics that we obtain from the cases in Theorem~\ref{thm: cases}, as shown in \cite{mathrepo}.\footnote{These 12 cases are presented in the section titled \emph{Decompostion of the Spohn Cubic} in \cite{mathrepo}.}
We may assume that the defining polynomials are always irreducible, as otherwise we obtain a union of two (possibly identical) lines. \\
(1) \& (3) The conic in this case is of the form 
$g = d_1 xy + d_2 xz + d_3 yz + d_4 z^2$
and its Jacobian at a point $p$ is given by
$$ J(p) =
\begin{pmatrix}
d_1 y + d_2 z &
d_1 x + d_3 z &
d_2 x + d_3 y + 2 d_4 z
\end{pmatrix}.
$$
Clearly, if $d_1 = 0$, then $g$ is reducible. Hence, we can assume that $d_1 \neq 0$, under which circumstance the point $[0:1:0]$ is a smooth point. \\
(2) The conic in this case is of the form
\begin{align*}
g &= (a_{21}-a_{22})(b_{11}-b_{12})x^2 + (a_{12}-a_{22})(b_{11}-b_{12})xy \\
&+ ((a_{21}-a_{12})(b_{11}-b_{22}) 
+(a_{21}-a_{22})(b_{21}-b_{12}))xz \\
&+ (a_{22}-a_{12})(b_{12}-b_{21})yz 
+ (a_{21}-a_{12})(b_{21}-b_{22})z^2 \\
&= d_1 x^2 + d_2 xy + d_3 xz + d_4 yz + d_5 z^2
\end{align*}
and its Jacobian at a point $p$ is given by
$$ J(p) =
\begin{pmatrix}
2d_1 x + d_2 y + d_3 z &
d_2 x + d_4 z &
d_3 x + d_4 y + 2 d_5 z
\end{pmatrix}.
$$
Assume for contradiction that $d_2=d_4=0$. This implies that $a_{12}=a_{22}$, in which case $g = (a_{21}-a_{12})(x+z)((b_{11}-b_{12})x+(b_{21}-b_{22})z)$, or that $b_{11}=b_{12}=b_{21}$, in which case $d_1=0$ and $g$ is divisible by $z$. Therefore, either $d_2$ or $d_4$ is non-zero and hence $[0:1:0]$ is a smooth point. \\
The cases (4),(5) and (6) follow in the same way as the previous three. \\
(7), (11) \& (12) The conic in this case is of the form
$g = d_1 xy + d_2 xz + d_3 yz$
and its Jacobian at a point $p$ is given by 
$$J(p) = 
\begin{pmatrix}
d_1 y + d_2 z &
d_1 x + d_3 z &
d_2 x + d_3 y 
\end{pmatrix}.
$$ 
Since $g$ is irreducible, $d_1$ must be non-zero and thus $[1:0:0]$ is a smooth point. \\
(8) The conic in this case is of the form
$ g = d_1 xy + d_2 y^2 + d_3 xz + d_4 z^2$
and its Jacobian at a point $p$ is given by
$$ J(p) =
\begin{pmatrix}
d_1 y + d_3 z &
d_1 x + 2 d_2 y &
d_3 x + 2 d_4 z 
\end{pmatrix}.
$$
If $d_1=d_3 = 0$, then $g$ is reducible over $\mathbb C$, hence we may assume that one of them is non-zero, under which circumstance the point $[1:0:0]$ is smooth. \\
(9) \& (10) The conic in this case is of the form
$g = d_1 x^2 + d_2 xy + d_3 xz + d_4 yz$
and its Jacobian at a point $p$ is given by 
$$J(p) = 
\begin{pmatrix}
2 d_1 x + d_2 y + d_3 z &
d_2 x + d_4 z &
d_3 x + d_4 y 
\end{pmatrix}.
$$ 
Since $g$ is irreducible, $d_4$ must be non-zero and thus $[0:1:0]$ is a smooth point. 
\end{proof}
\end{lemma}

Our aim is to extend these properties to the Spohn variety $\V \subseteq \mathbb P^3$. 
For generic games, this has already been done in \cite[Proposition 4.1]{PortakalWindisch}, but we are interested in the non-generic cases from Theorem~\ref{thm: cases}. 
We start by decomposing $\V$ into its \href{https://mathrepo.mis.mpg.de/elliptic_curves_game_theory/#decomposition-of-the-spohn-variety}{\textcolor{magenta}{irreducible components}} in the twelve cases and under the assumption that all the other entries are generic. This can be found in \cite{mathrepo}.

\begin{proposition}\cite[Proposition 5.8]{Baldi}
\label{smoothrealpoint}
Let $V \subset \mathbb{C}^m$ be an irreducible variety defined over $\mathbb R$ of dimension $d$ and $\pi: \mathbb{C}^m \to \mathbb{C}^{d+1}$ be a generic projection defined over $\mathbb{R}$. If $\overline{\pi(V)}$ contains a real smooth point, then so does $V$.
\end{proposition}

\begin{theorem}\label{smooth real point in V}
For each of the cases from Theorem~\ref{thm: cases}, assuming that all the other entries of the payoff tables are generic, the irreducible components of the Spohn variety $\V$ each contain a smooth point with real coordinates, hence the set of real points of $\V$ lies Zariski dense in $\V$.
\begin{proof}

Consider an irreducible component $V = \mathbb V_{\mathbb{P}^3}(P)$ of $\V$, where $P$ is prime. If $W:= \mathbb V_{\mathbb P^2}(P \cap \mathbb C[p_{11},p_{12},p_{21}])$ is an irreducible component of $\mathcal C$, then, under our assumptions, it contains a smooth real point $p$. By covering $\mathbb P^2$ with the affine open spaces $\widetilde{U}_{p_{11}=1}$, $\widetilde{U}_{p_{12}=1}$ and $\widetilde{U}_{p_{21}=1}$, for illustration we assume here that $p \in \widetilde{U}_{p_{11}=1}$. 
Let $U_{p_{11}=1}$ be the corresponding affine open space of $\mathbb{P}^3$ and $\pi: \mathbb C^3 \to \mathbb C^2$ the generic projection between affine spaces. Then
$$ V \cap U_{p_{11}=1} \simeq \mathbb V_{\mathbb C^3}(g(1,\cdot) \mid g \in P )$$
and
$$\overline{\pi(\mathbb V_{\mathbb C^3}(g(1,\cdot) \mid g \in P ))} = \mathbb V_{\mathbb{C}^2} (g(1,\cdot) \mid g \in P \cap \mathbb C[p_{11},p_{12},p_{21}]) \simeq W \cap \widetilde{U}_{p_{11}=1}.
$$
If now $\dim V = 1$, then by Proposition~\ref{smoothrealpoint}, the preimage $V \cap U_{p_{11}=1}$ also contains a smooth real point. \\
Considering the specific cases, we start by looking at the \href{https://mathrepo.mis.mpg.de/elliptic_curves_game_theory/#proof-of-theorem-3-7}{\textcolor{magenta}{decompositions of $\V$}} (\cite{mathrepo}). One can check in all cases that every irreducible component has codimension 2. For all cases except (7), one can also check that the elimination ideals of the minimal primes are exactly the minimal primes of $\mathcal C$ for that specific case. By the above reasoning, this implies the existence of a real smooth point in the components of $\V$. In case (7), the components $\mathbb V(p_{21},p_{11})$ and $\mathbb V(p_{12},p_{11})$ remain the same after eliminating $p_{22}$. They are proper subsets of the component $\mathbb V(p_{11})$ of $\mathcal C$. However, since they are lines, they must contain a smooth real point.
Elimination from the third component gives the remaining component of $\mathcal C$ in that case, so as before, this gives us a real smooth point in said component of $\V$.
\end{proof}
\end{theorem}

\noindent If the cubic $\mathcal C \subseteq \mathbb P^2$ is reducible, then so is the Spohn variety $\V \subseteq  \mathbb P^3$. There are, however, cases in which $\V$ is reducible and $\mathcal C$ is not. 

\begin{example}\label{counterexample}
Assume that $a_{12}=a_{22}$ and that all other entries are generic. This is not among the cases from Theorem~\ref{thm: cases}, hence, $\mathcal C = \mathbb V(f)$ is irreducible. We can, however, \href{https://mathrepo.mis.mpg.de/elliptic_curves_game_theory/#example-3-8}{\textcolor{magenta}{compute}} that 
\begin{align*}
\mathbb I(\V) = (p_{21},p_{11}) \cap (\det M_1, \det M_2, 
((a_{21}-a_{22})b_{11}+ (a_{22}-a_{21})b_{21})p_{12}^2 \\
+ (a_{21}-a_{11})b_{12}+(a_{11}-a_{21})b_{21})p_{12}p_{21} \\
+ ((a_{21}-a_{22})b_{11} + (a_{22}-a_{11})b_{12} + (a_{11}-a_{22})b_{21} + (a_{22}-a_{21})b_{22}) p_{12}p_{22} \\
+ ((a_{11}-a_{21})b_{21} + (a_{21}-a_{11})b_{22})p_{21}p_{22} \\
+ ((a_{11}-a_{22})b_{21} + (a_{22}-a_{11})b_{22})p_{22}^2)
\end{align*}
Eliminating $p_{22}$ from these gives the ideals $(p_{21},p_{11})$ and $(f)$. It is $V := \mathbb V(p_{21},p_{11}) \subsetneq \mathbb V(f) = \mathcal C$. This corresponds to $\mathcal C$ being irreducible and also, one can observe that we cannot simply obtain a real smooth point in $V$ by pulling back a real smooth point in the planar model. Since $V$ is a line, it contains a real smooth point anyways; therefore, the real points also lie dense in $\mathcal C$ in this case.
\end{example}

\noindent Similarly, $\V$ is reducible if $b_{21}=b_{22}$. There are no other known cases where $\V$ is reducible but $\mathcal C$ is not. \\

Under the cases of Theorem~\ref{thm: cases}, the cubic $\mathcal C$ is reducible if and only if the defining equation $f$ is non-zero. Similarly, the conditions for $\V$ being reducible are also not actually closed. For example, if $a_{11}=a_{21} \neq a_{12}=a_{22}$ and  $b_{11}=b_{12} \neq b_{21}=b_{22}$, then $\V = V(p_{11}p_{22} + p_{12}p_{21})$ is irreducible. Notice that this also results from $f$ being zero in this case. \\

Tracing back from the reducibility and the denseness of real points of the Spohn curve to what this means in practice, it is natural to ask if in the cases when the Spohn curve is reducible, the probabilities that lie on it are actually dependency equilibria.

\begin{remark}
Following \cite{PortakalWindisch}, denote 
$$\mathcal W := \mathbb V((p_{11}+p_{12})(p_{21}+p_{22})(p_{11}+p_{21})(p_{12}+p_{22})).$$ Then, any point $p \in \overline{\Delta} \cap \overline{(\V \setminus \mathcal W)}^{\text{Zar}}$ is a dependency equilibrium. For any point contained in the union of hyperplanes $\mathcal W$, that is not necessarily the case. 
\cite[Theorem 4.9]{PortakalWindisch} 
proves that, 
assuming all other entries of the payoff tables are generic,
in the cases (1) to (7) from Theorem~\ref{thm: cases}, it is $\overline{\Delta} \cap \overline{(\V \setminus \mathcal W)}^{\text{Zar}} \subsetneq \overline{\Delta} \cap \V$, which means we cannot make general statements about dependency equilibria from the Spohn variety in these cases. It also shows that in the cases (8) to (12), equality holds and that, therefore, any point $p \in \overline{\Delta} \cap \V$ is a dependency equilibrium. 
Indeed, one can see in the \href{https://mathrepo.mis.mpg.de/elliptic_curves_game_theory/#remark-3-9}{\textcolor{magenta}{decomposition of $\V$}} in \cite{mathrepo} that, in the cases (1) to (7), some components of $\V$ are contained in $\mathcal W$ and that, in the cases (8) to (12), no component of $\V$ is entirely contained in $\mathcal W$. Furthermore, in these latter cases, \cite[Corollary 3.20]{PortakalWindisch} also guarantees that every Nash equilibrium of the game is also a dependency equilibrium.
\end{remark}


\section{Game theorist's guide to elliptic curve invariants}
\label{sec:equivalentgames}
We have studied in detail when the Spohn curves of $2 \times 2$ games are reducible and therefore not smooth. We will now take a closer look at Spohn curves from generic $2 \times 2$ games, for which we already know that real points lie dense in them (\cite[Proposition 4.1]{PortakalWindisch}) and that the points in $\overline{\Delta} \cap \V$ are dependency equilibria (\cite[Corollary 3.20]{PortakalWindisch}). Also, by Theorem~\ref{parametrization}, these Spohn varieties are elliptic curves and we may use properties of elliptic curves to make statements about dependency equilibria of generic $2\times 2$ games.

It is natural to ask when two generic games are equivalent with respect to their Spohn curves. {This can be addressed using invariants of elliptic curves i.e., by computing the $j$-invariant \eqref{eq:jinvt} and then determining if they are isomorphic over the field of definition}. 
We review the calculation of the $j$-invariant of an elliptic curve given by the intersection of two quadrics in $\mathbb P^3$ in the general case where the quadrics are defined over $\mathbb Q$, and we provide examples including cases where elliptic curves arise as Spohn varieties.
\subsection{Review of invariants of elliptic curves}
    {In this section, we provide the reader with necessary background to compute the $j$-invariant of elliptic curves. \emph{The aim of this section is purely to serve as a primer to readers whose background does not include arithmetic of elliptic curves}.  We restrict to elliptic curves over $\mathbb Q$ since in many known examples of games, the payoffs are rational. Should this not be the case, payoffs can be approximated by rational numbers (for example, via the method of continued fractions \cite[§ 1.8]{das}). See Appendix~\ref{app:contfrac} for more details.}
\begin{definition}[Elliptic curve over $K$]
    An elliptic curve over a field $K$, (char($K) \neq 2$) is a smooth (projective) curve of genus $1$ with at least one $K$-rational point.
\end{definition}
As mentioned above, we restrict ourselves to the case of $K = \mathbb Q$. Elliptic curves over $\mathbb Q$ can be expressed in the form of a cubic equation known as the \emph{Weierstrass form}. 
\begin{theorem}[\cite{Silverman}, Chapter III.2.3]
Every elliptic curve over $\mathbb Q$ can be expressed as a cubic in $\mathbb P^2$ with the following form: 
\begin{equation}\label{eq:cubicp2}
    E_\mathbb{Q} : = y^2z + a_1 xyz + a_3 yz^2 = x^3 + a_2 x^2z + a_4 xz^2 + a_6z^3, \ \ a_{1,2,3,4,6} \in \mathbb Q 
\end{equation}
which upon dehomogenization gives the more recognizable long Weierstrass form of the curve: 
\begin{equation*}
    E_\mathbb{Q} : = y^2 + a_1 xy + a_3 y = x^3 + a_2 x^2 + a_4 x + a_6, \ \ a_{1,2,3,4,6} \in \mathbb Q ~.
\end{equation*}
\end{theorem}
\begin{remark}[Short Weierstrass form, \cite{Silverman}, Chapter III.1]
Every elliptic curve over $\mathbb Q$ can be reduced further into the short Weierstrass form (via appropriate coordinate transformations)  which is given by 
\begin{equation*}
    E_\mathbb{Q} : = y^2 = x^3  + A x + B, \ \ A,B \in \mathbb Q  ~.
\end{equation*}
\end{remark}
Equation \eqref{eq:cubicp2} can be obtained from the intersection of two quadrics in $\mathbb P^3$, with  intersection at a \emph{known} $\mathbb Q$-point. 
\subsection{Invariants for Spohn curves as  intersection of quadrics in $\mathbb P^3$}
Consider two quadric equations $P_1, P_2 \in \mathbb Q[x,y,z,t]$ in $\mathbb P^3$. In the context of $2\times 2$ games, these quadrics are the exact same equations that appear in \eqref{eq:quadricexample}  after renaming the variables. We require that these quadrics intersect in $\mathbb P^3$ at an arbitrary rational point $k = [x_0: y_0: z_0: t_0]$. It is important to know what this point is, and in all the cases at hand we are given this rational point. Finding rational points on elliptic curves is an interesting and challenging endeavor in its own right. We do not make any comments on finding a common rational solution here.\footnote{In addition to the challenging nature of this problem, we also have that for Spohn curves, we know that e.g.\ $k = [0:0:0:1]$, i.e.\ the point at infinity in $\mathbb P^3$, is a rational solution.}  Now, given two such intersecting quadrics with an arbitrary common rational point, we wish to compute the $j$-invariant. 

The first step is an appropriate representation of quadrics in $\mathbb P^3$. Let the coordinates in $\mathbb P^3$ be $V = (x,y,z,t)^\top$.\footnote{These are projective coordinates but for the sake of explaining the algorithm, we represent them as a vector.} We represent the quadrics as 
\begin{equation*}
\begin{split}
     P_1 = V^\top \cdot A \cdot V, \hspace{0.3cm} \ A  = &\begin{pmatrix}
        a_{00} & a_{01} & a_{02} & a_{03} \\ 
        a_{10} & a_{11} & a_{12} & a_{13} \\ 
        a_{20} & a_{21} & a_{22} & a_{23} \\ 
        a_{30} & a_{31} & a_{32} & a_{33} \\ 
    \end{pmatrix},~ \hspace{0.3cm} A \in \text{ Mat}_{4\times 4}(\mathbb Q) \\
        P_2 = V^\top \cdot B \cdot V, \hspace{0.3cm} \ B  = &\begin{pmatrix}
        b_{00} & b_{01} & b_{02} & b_{03} \\ 
        b_{10} & b_{11} & b_{12} & b_{13} \\ 
        b_{20} & b_{21} & b_{22} & b_{23} \\ 
        b_{30} & b_{31} & b_{32} & b_{33} \\ 
    \end{pmatrix},~ \hspace{0.3cm} B \in \text{ Mat}_{4\times 4}(\mathbb Q) ~. 
\end{split}
\end{equation*}
Given $k = [x_0: y_0: z_0: t_0]$, the common rational solution to the two quadrics $P_1, P_2$, we define 
\begin{equation}
\label{eq:solnmatrix}
    Q = \begin{pmatrix}
        1 & 0 & 0 & x_0 \\
        0 & 1 & 0 & y_0 \\
        0 & 0 & 1 & z_0 \\ 
        0 & 0 & 0 & t_0
    \end{pmatrix}~.
\end{equation}
When the common rational solution is $[0:0:0:1]$, the matrix $Q$ as in \eqref{eq:solnmatrix} is just the $4\times 4$ identity matrix. This is always the case we have when considering the Spohn curve. 
Having the common rational point at infinity makes it easier to project the curve to lower dimensions. 
\\[0.3cm]
In case the common rational solution of the 
intersection of the two quadrics is \emph{not} the point at infinity i.e.\ $Q$ is not the $4\times 4$ identity matrix, one can perform coordinate transformations on $P_1$ and $P_2$ such that the common rational solution becomes the point at infinity. This straightforward algorithm is explained below and also in \cite{knaf}. To transform the two quadrics such that the common solution is the point at infinity:   \vspace{0.5cm}
\begin{enumerate}[label=(\textbf{Step \Roman*})]
\item \textbf{Normalize solutions:} Choose a normalization of the solutions such that $t_0\neq 0$. One may also assume that $t_0 = 1$ as in \cite{knaf}. \vspace{0.5cm}
    \item \textbf{Coordinate change:} Now, let $V=(x,y,z,1)^{\top}$ be the coordinates of $\mathbb P^3$. Transform to coordinates $W = (X, Y, Z, T)^\top$ such that the common rational point in the new coordinates is at $[X_0:Y_0:Z_0:T_0] = [0:0:0:1]. $ This is done by the following transformation:
    \begin{equation*}
        V = \begin{pmatrix}
        1 & 0 & 0 & x_0 \\
        0 & 1 & 0 & y_0 \\
        0 & 0 & 1 & z_0 \\ 
        0 & 0 & 0 & 1
    \end{pmatrix} \cdot W , \hspace{0.2cm} \ W = \begin{pmatrix}
        1 & 0 & 0 & -x_0 \\
        0 & 1 & 0 & -y_0 \\
        0 & 0 & 1 & -z_0 \\ 
        0 & 0 & 0 & 1
    \end{pmatrix} \cdot V~.
    \end{equation*}
    \vspace{0.2cm}
    \item \textbf{Change of quadrics:} Under the coordinate change, the quadrics transform as follows
\begin{equation*}
\begin{split}
  P_1 (x,y,z,t) \mapsto  P_1'(X,Y,Z,T) &= W^\top\cdot Q^\top \cdot A \cdot Q \cdot W \\ 
 P_2 (x,y,z,t) \mapsto  P_2'(X,Y,Z,T) &= W^\top\cdot Q^\top \cdot B \cdot Q \cdot W ~.
\end{split} 
\end{equation*}
\vspace{0.2cm}
 \item \textbf{Eliminating terms:} This follows from \cite[§ 7.2.2]{cohen1}. 
    We now have two quadrics $P_1'[X,Y,Z,T]$ and $P_2'[X,Y,Z,T] \subseteq \mathbb P^3 $ with a common rational point at $[0:0:0:1]$. Both these quadrics can be expressed as  
    \begin{equation*}
        P_i'(X,Y,Z,T) = K_i \ T^2 + L_i(X,Y,Z) \ T + M_i(X,Y,Z), \ i = 1,2 
    \end{equation*}
    where $K_i$ is a constant, $L_i(X,Y,Z)$ are linear in $X,Y,Z$ and $M_i$ are quadratic in $X,Y,Z$. Since the common rational point is already at $[0:0:0:1]$, $K_i = 0$ for $i = 1,2$. This means that the two intersecting quadrics take the form 
    \begin{equation*}
        \begin{split}
             P_1'(X,Y,Z,T) &=  L_1(X,Y,Z) \  T + M_1(X,Y,Z) \\
              P_2'(X,Y,Z,T) &=  L_2(X,Y,Z) \ T + M_2(X,Y,Z) ~.
        \end{split}
    \end{equation*}
It is important to note that $L_1$ and $L_2$ above are linearly \emph{independent}. The proof of this statement can be found in \cite[§ 7.2.2]{cohen1}. To state it succinctly, if $L_1, L_2$ are linearly \emph{dependent}, then there is a linear combination which will make the intersecting quadric a genus $0$ curve, which is false for elliptic curves. 
\vspace{0.5cm}
\item \textbf{Reduction to cubic in $\mathbb P^2$}. The equation $$C(X,Y,Z) = L_1(X,Y,Z) M_2(X,Y,Z) - L_2(X,Y,Z) M_1(X,Y,Z), \ C\subseteq \mathbb P^2 $$ is homogeneous of degree 3, and represents an elliptic curve in homogeneous, projective coordinates. One may wish to proceed to represent the curve in Weierstrass form from here, however it is not necessary for the purposes of this paper.
\end{enumerate}
\begin{remark}[Defining elliptic curves from cubics in \texttt{Pari/GP}] In \texttt{Pari/GP} \cite{pari}, one can use in-built commands \texttt{ellfromeqn} and \texttt{ellinit} to define elliptic curves directly from the cubic in $\mathbb P^2$. The file \href{https://mathrepo.mis.mpg.de/elliptic_curves_game_theory/#computations-for-section-4}{\textcolor{magenta}{\texttt{getJ.gp}}} in \cite{mathrepo} contains the relevant code to do this in \texttt{Pari/GP}.
\end{remark}
\begin{remark}
   To re-iterate, since $[0:0:0:1]$ is always a common rational solution for Spohn curves, one may proceed straight to (\textbf{Step IV}) above. \end{remark}
\textbf{Aronhold invariants and the $j$-invariant:}

Starting with two quadrics that intersect at an arbitrary rational point, we now have a cubic in $\mathbb P^2$. This cubic depends on the monomials in the two quadrics, as well as the common rational point in $\mathbb P^3$, which has been transformed to $[0:0:0:1]$. A generic rational cubic equation in 3 variables $(x,y,z)$\footnote{We reset to lowercase coordinates for ease of readability henceforth. For the case of Spohn curves, the coordinates $(x,y,z,t) = (X,Y,Z,T)$.} will have the following form 
\begin{equation}
\label{eq:cubic}
   C(x,y,z):= a x^3 + b y^3 + cz^3 + 3dx^2 y + 3ey^2z + 3fz^2x + 3gxy^2 + 3hyz^2 + 3izx^2 + 6jxyz
~,\end{equation}
where the coefficients $a,b,c,d,e,f,g,h,i,j\in \mathbb Q$. 
From the coefficients of the cubic (and as a result, the coefficients of the monomials of $P_1, P_2$ and the common rational point $[x_0: y_0: z_0: t_0]$), we may derive the two Aronhold ($S$ and $T$) invariants of the cubic $C(x,y,z)$. We wish to point out that the Aronhold invariants as in \eqref{eq:aronholdS} and \eqref{eq:aronholdT} are equivalent to the expressions given in  \cite{mchan} and \cite{sturmfels} under the rescaling of the coefficients of the cubic.\footnote{Since in these references, the cubic is represented as $C(x,y,z) = a x^3 + b y^3 + cz^3 + dx^2 y + ey^2z + fz^2x + gxy^2 + hyz^2 + izx^2 + jxyz$, i.e.\ without the factors of $3$'s and $6$.}
\\[0.2cm]
For the cubic equation of the form \eqref{eq:cubic},
the $S$ invariant is given as:
\begin{equation}
\label{eq:aronholdS}
\begin{split}
        S &= agec - agh^2 - ajbc + ajeh + afbh - af e^2 - d^2ec + d^2h^2 + dibc  \\ & - dieh + dgjc - dgf h - 2dj^2h + 3djfe - df^2b - i^2bh + i^2e^2 - ig^2c  \\ &+ 3igjh - igfe - 2ij^2e + ijfb + g^2f^2 - 2gj^2f + j^4~,
\end{split}
\end{equation}
and the $T$ invariant is given by:

\allowdisplaybreaks
\begin{align}
\label{eq:aronholdT}
\begin{split}
T &= a^2b^2c^2 - 3a^2e^2h^2 - 6a^2behc + 4a^2bh^3 + 4a^2e^3c - 6adgbc^2 \\ 
&+ 18adgehc - 12adgh^3 + 12adjbhc - 24adje^2c + 12adjeh^2 \\ 
&- 12adfbh^2 + 6adfbec + 6adfe^2h + 6aigbhc - 12aige^2c + 6aigeh^2 \\ 
&+ 12aijbec + 12aije^2h - 6aifb^2c + 18aifbeh - 24ag^2jhc - 24aijbh^2 \\ 
&- 12aife^3 + 4ag^3c^2 - 12ag^2fec + 24ag^2fh^2 + 36agj^2ec + 12agj^2h^2 \\ 
&+ 12agjfb c - 60agjfeh - 12agf^2bh + 24agf^2e^2 - 20aj^3bc - 12aj^3eh \\ 
&+ 36aj^2fbh + 12aj^2fe^2 - 24ajf^2be + 4af^3b^2 + 4d^3bc^2 - 12d^3ehc \\ 
&+ 8d^3h^3 + 24d^2ie^2c - 12d^2ieh^2 + 12d^2gjhc + 6d^2gfec - 24d^2j^2h^2 \\ 
&- 12d^2ibhc - 3d^2g^2c^2 - 24g^2j^2f^2 + 24gj^4f - 12d^2gfh^2 + 12d^2j^2ec \\ 
&- 24d^2jfb c - 27d^2f^2e^2 + 36d^2jfeh + 24d^2f^2bh + 24di^2bh^2 \\ 
&- 12di^2bec - 12di^2e^2h + 6dig^2hc - 60digjec + 36digjh^2 + 18digfb c \\ 
&- 6digfeh + 36dij^2bc - 12dij^2eh - 60dijfbh + 36dijf e^2 + 6dif^2be \\ 
&+ 12dg^2jfc - 12dgj^3c - 12dgj^2fh + 36dgjf^2e - 12dgf^3b + 24dj^4h \\ 
&+ 12dj^2f^2b + 4i^3b^2c + 24i^2g^2ec - 27i^2g^2h^2 - 36dj^3fe - 12i^3beh \\ 
&+ 8i^3e^3 - 24i^2gjbc + 36i^2gjeh + 6i^2gfbh + 12i^2j^2bh - 3i^2f^2b^2 \\ 
&- 12dg^2f^2h - 12i^2gfe^2 - 24i^2j^2e^2 + 12i^2jfbe - 12ig^3fc + 12ig^2j^2c \\ 
&+ 36ig^2jfh - 12ig^2f^2e - 36igj^3h - 12igj^2fe + 12igjf^2b + 24ij^4e \\ 
&- 12ij^3fb + 8g^3f^3 - 8j^6 ~.
\end{split}
\end{align}

\noindent These invariants can also be found at \cite{zeb}.
 The discriminant of the elliptic curve is 
    \begin{equation}
    \label{eq:disc}
        \Delta = \frac{64 S^3 - T^2}{1728}~.
    \end{equation}
The $j$-invariant of the elliptic curve is 
    \begin{equation}
    \label{eq:jinvt}
        j = \frac{64 S^3}{\Delta}~.
    \end{equation}
    The $j$-invariant and discriminant as above are again equal to the corresponding expressions in \cite{sturmfels} under the change in notation for the coefficients of the cubic.
    The $j$-invariant in terms of the coefficients of the two quadrics in $\mathbb P^3$ that intersect at an arbitrary rational point is available in \href{https://mathrepo.mis.mpg.de/elliptic_curves_game_theory/#computations-for-section-4}{\color{magenta}\texttt{Generic\textunderscore J\textunderscore invt.txt}}\footnote{This file does not explicitly set $t_0 = 1$ as is assumed in Step 1 of the algorithm to transform the quadrics to have a common rational solution at infinity.}   at \cite{mathrepo}. The $j$-invariant for Spohn curves as in Section~\ref{sec:depeq} is given in \href{https://mathrepo.mis.mpg.de/elliptic_curves_game_theory/#computations-for-section-4}{\color{magenta}\texttt{SpohnCurveJInv.txt}}, also available at \cite{mathrepo}. It corresponds to the result in \cite[Proposition 12]{PortakalSturmfels}
\begin{theorem}[{\cite[Chapter III, Proposition 1.4]{Silverman}}]
\label{thm:isomorphismec}
    Two elliptic curves are isomorphic over $\bar{\mathbb{Q}}$ if and only if their $j$-invariants are equal.  
\end{theorem}

\begin{remark}
    Over $\mathbb{Q}$, not all elliptic curves with the same $j$-invariant are isomorphic, but isomorphic elliptic curves have the same $j$-invariant. {See Example~\ref{ex:bachstrav} for an example where elliptic (Spohn) curves over $\mathbb Q$ have the same $j$-invariant and are isomorphic over $\mathbb Q$, and see Example ~\ref{ex:nonQisom} for elliptic curves over $\mathbb Q$ which have the same $j$-invariant but are not isomorphic over $\mathbb Q$.}
\end{remark}
\textbf{$j$-invariants of Spohn curves: } {We now present a few examples to compute the $j$-invariant of elliptic curves obtained by intersecting quadrics in $\mathbb P^3$ using the steps explained in this section. We consider three classes of examples. The first for the Spohn curve of a generic $2\times 2$ game, the second for a non-generic game, and the third for an elliptic curve that is not a Spohn curve.}
\begin{example}[Generic $2\times 2$ game]\label{Examplejinv}
Consider the quadrics
\begin{equation*}
    \begin{split}
        P_1 &=  -xz+2xt-2yz+yt =  t(2x +y) -(xz + 2yz) \\ 
        P_2 &= -5yx-6xt-3yz-4zt = -t(6x + 4z) -(5xy +3yz) 
    \end{split}
\end{equation*}
coming from the $2 \times 2$ game with payoff tables 
$X^{(1)} = \begin{pmatrix}
    1&2 \\ 0&3
\end{pmatrix}$
and $X^{(2)} = \begin{pmatrix}
    6&1 \\ 4&0
\end{pmatrix}$.
The point $[0:0:0:1]$ is a common rational solution, so we have no terms quadratic in $t$ in the quadrics. To reduce to a cubic in $\mathbb P^2$, we have
\begin{equation*}
\begin{split}
        L_1 &=  (2x +y), \qquad \hspace{0.85cm} L_2 = -(6x + 4z) \\
    M_1 &= -(zx + 2zy), \qquad M_2= -(5yx + 3zy)~,
\end{split}
\end{equation*}
and we recall that $L_i$ is the term in $P_i$ that is linear in $t$, and $M_i$ is the term in $P_i$ that is independent of $t$.
The reduced cubic in $\mathbb P^2$ is 
 \begin{equation*}
    C(x,y,z) = L_1M_2 - L_2 M_1,
 \end{equation*}
 which is explicitly given by
 \begin{equation*}
  C:=  \left(-10y
 - 6z\right) x^2
 + \left(-5y^2
 - 18zy
 - 4z^2\right) x
 + \left(-3zy^2
 - 8z^2y\right) ~.
 \end{equation*}
 Comparing with the coefficients of \eqref{eq:cubic}, we can compute the two Aronhold invariants $S$ \eqref{eq:aronholdS} and $T$ \eqref{eq:aronholdT}, and use them to compute the $j$-invariant as:
 \begin{equation*}
    j = \dfrac{2810381476}{227025}.
 \end{equation*}
This can be computed using the \href{https://mathrepo.mis.mpg.de/elliptic_curves_game_theory/#functions-in-the-mathematica-notebook}{\textcolor{magenta}{\texttt{IntersectionQuadricsJ.nb}}} in \texttt{Mathematica}. One can also initialize the curve in \texttt{Pari/GP} using the \href{https://mathrepo.mis.mpg.de/elliptic_curves_game_theory/#computations-for-section-4}{\textcolor{magenta}{\texttt{getJ.gp}}} script and find the same value for the $j$-invariant (\cite{mathrepo}). An immediate observation here is that for any $\lambda_1, \lambda_2 \in \mathbb R \backslash \{0\}$ and $\alpha_1, \alpha_2 \in \mathbb R$, the games defined by the following two payoff matrices $$\begin{pmatrix}
    \lambda_1  + \alpha_1 &2 \lambda_1 + \alpha_1 \\ \alpha_1 &3\lambda_1 +\alpha_1
\end{pmatrix}
\text{ and }  \begin{pmatrix}
    6\lambda_2+\alpha_2 & \lambda_2+\alpha_2 \\ 4 \lambda_2 + \alpha_2&\alpha_2
\end{pmatrix}$$
have the same (elliptic) Spohn curve and thus the same $j$-invariant. 
\end{example}
\begin{example}[Non-generic $2\times 2$ game]
Consider the quadrics 
\begin{align*}
    P_1(x,y,z,t) &= xz + yz - xt - yt =  - (x+y)t + (xz + yz) \\
    P_2(x,y,z,t) &= -5xy - yz + xt + 5zt =   t (x + 5z) -(5xy + yz)~.
\end{align*}
coming from the $2 \times 2$ game with payoff tables 
$X^{(1)}= \begin{pmatrix} 1& 1 \\ 2 &0 \end{pmatrix}$
and $X^{(2)} = \begin{pmatrix} 3 & -2 \\ -1 & 4 \end{pmatrix}$, for which $[0:0:0:1]$ is a common solution. By following the algorithm above, we can eliminate $ t $ to get a cubic: 
\begin{align*}
    C(x,y,z) &= -(xz + yz)(x + 5z)+(x+y)(5xy + yz)~. 
\end{align*}
The $j$-invariant of this cubic is infinite i.e.\ the discriminant $\Delta$ as in \eqref{eq:disc} is zero and the cubic, being singular, therefore does not define an elliptic curve. This is also apparent from the fact that the game is not generic and falls under case (1) of Theorem~\ref{thm: cases}.
\end{example}
{Intersecting quadrics that arise from generic $2\times 2$ games have $[0:0:0:1]$ as a common rational point. However, as a quick remark, one may use the algorithm explained in this section to compute the $j-$invariant of quadrics that do not necessarily intersect at $[0:0:0:1]$, but rather a different $\mathbb Q$-rational point. For example:}
\begin{example}[Non-Spohn elliptic curve]
Consider the quadrics   
\begin{align*}
    P_1(x,y,z,t) &= x^2+y^2 - z^2 - t^2 \\
    P_2(x,y,z,t) &= xz - zy + yt - zt~.
\end{align*}
for which a common rational solution is $[1:1:1:1]$, but not $[0:0:0:1]$. By following the algorithm as explained {in Steps (I - V) earlier in this section}, we can reduce the two quadrics to a cubic in $\mathbb P^2$ which has the form 
\begin{equation*}
    C(x,y,z) = -x^3 - x y^2 + 3 x^2 z - y^2 z - x z^2 + 2 y z^2 - z^3 ~.
\end{equation*}
This represents an elliptic curve, but does not correspond to the Spohn curve of any generic $2\times 2$ game since it does not have the form (Equation~\ref{eq: Spohn cubic}) of the Spohn cubic. The $j$-invariant of this curve is
\begin{equation*}
    j = \frac{65536}{37} ~.
\end{equation*}
\end{example}

\textbf{An equivalence definition of $2\times 2$ games: }
Here we provide a characterization of equivalent games in terms of the $j$-invariants. 
We recapped in the previous subsection that if elliptic curves over $\mathbb Q$ have the same $j$-invariant, it does not guarantee that they are isomorphic over $\mathbb Q$, and they could be isomorphic over a quadratic extension of $\mathbb Q$. However, it is easy to check if whether two curves defined over the same field are indeed isomorphic. This is implemented in \texttt{Pari/GP} via them command \texttt{ellisisom}. We refer the reader to the \texttt{Pari/GP} manual for more details regarding this. 
We propose a {definition} of full equivalence with respect to the Spohn curve which we define as follows: 
\begin{definition}
 Two $2\times 2$ games whose payoff matrices are $\mathbb Q$-valued are \emph{fully equivalent with respect to the Spohn curve} if and only if the two Spohn curves over $\mathbb Q$ are isomorphic. 
\end{definition}

\begin{example}[Bach or Stravinsky perturbed]
\label{ex:bachstrav}
We consider the following perturbed Bach or Stravinsky games, where two players try to decide on a concert together. They value attending the concert together more than going to their individual preferred concerts. The games are represented in bimatrix form, with the payoff of one player for jointly choosing Bach or Stravinsky being one.
\[
\begin{array}{c|cc}
   & B & S \\
  \hline
  B & (3,2) & (0,1) \\
  S & (0,0) & (2,3) \\
\end{array}
\quad
\begin{array}{c|cc}
   & B & S \\
  \hline
  B & (3,2) & (1, 0) \\
  S & (0,0) & (2,3) \\
\end{array}
\quad
\begin{array}{c|cc}
   & B & S \\
  \hline
  B & (3,2) & (0,0) \\
  S & (0,1) & (2, 3) \\
\end{array}
\quad
\begin{array}{c|cc}
   & B & S \\
  \hline
  B & (3,2) & (0, 0) \\
  S & (1,0) & (2,3) \\
\end{array}
\]

\noindent We obtain the Spohn cubics corresponding to these four games in order as follows:
\begin{align*}
&C_1 =  x^2y - 2xy^2 + 3x^2z - xyz + 2y^2z+ 9xz^2 \\
&C_2 = 2x^2y - 2xy^2 + 3x^2z + xyz + 3yz^2+ 9xz^2 \\
&C_3 = 2x^2y - 4xy^2 + 3x^2z + xyz - 2y^2z + 6xz^2 \\ 
&C_4 = 2x^2y - 4xy^2 + 2x^2z - xyz - 3yz^2 + 6xz^2 
\end{align*}

The $j$-invariant of these Spohn cubics are all the same and is equal to $365986170577/44976384$. In fact, these curves are all isomorphic to each other over $\mathbb Q$. We therefore conclude that the games corresponding to $C_1, C_2, C_3, C_4$ are fully equivalent.
\end{example}

{To re-iterate}, it is important to point out that two elliptic curves with the same $j$-invariant need not be isomorphic over $\mathbb Q$. In such cases, the two elliptic curves can be isomorphic up to \emph{twisting by a character} \cite[§10.5]{Silverman}. 
\begin{example}
\label{ex:nonQisom}
Consider the two elliptic curves in Weierstra\ss~ form 
\begin{align*}
    E_1 :  &-x^3 - 103072987022928/199086408481 \ x + \\ &(y^2 - 52977693274235725360768/88830563686545871)
\end{align*} and 
\begin{align*}
    E_2: &-x^3 - 2576824675573200/199086408481 \ x + \\ & (y^2 - 6622211659279465670096000/88830563686545871)~. 
\end{align*}
The $j$-invariant of both these curves is $44564/446191$. However, these curves are not isomorphic over $\mathbb{Q}$.\footnote{In fact $E_2$ is a quadratic twist of $E_1$ by a character $\bmod\  5$.}
\end{example}
We have thus far not encountered any examples of two Spohn cubics which have the same $j$-invariant but are not isomorphic over $\mathbb Q$, and are related to each other by a twist. It is of our opinion that these cases demand further careful scrutiny in the context of Spohn cubics. Such a study would be necessary in order to fully characterize when two $2\times 2$ games are equivalent and what this means in terms of dependency equilibria.

\appendix
\section{Continued fraction approximation of real numbers}
\label{app:contfrac}
In \texttt{Pari/GP}, rational approximations can be implemented using the command \texttt{contfrac}. A \href{https://mathrepo.mis.mpg.de/elliptic_curves_game_theory/#rational-approximations-of-real-numbers}{\textcolor{magenta}{script }}to compute rational approximations using continued fractions (\texttt{confrac.gp}) can also be found on \cite{mathrepo}. A short tutorial and example are presented below:

\begin{lstlisting}[caption=Evaluating Continued Fractions, label={lst:contfrac}, ]
    \\ Continued fraction representation of a number x to k convergents
    >> contfrac(x,k)

    \\ Find rational approximation of the number
     >> eval_contfrac(V) = { local(x=0); 
                        forstep(i = #V, 1, -1, x = V[i] + if(x,1/x,0)); 
                        return(x);}
    \\ To get decimal representation, use `.'
\end{lstlisting}
\begin{example}[Rational approximation to $\zeta(3)$] Irrationality of $\zeta(3)$ was established in the 1970's by R. Ap\'ery, meaning that its continued fraction representation is infinitely long. However, one can approximate $\zeta(3)$ as a rational number to any finite precision using continued fractions.
\end{example}

\begin{lstlisting}[caption=Rational approximation of $\zeta(3)$, label={lst:exzeta3}]
    \\Examples: Consider zeta(3) which is an irrational number
    >> zeta(3)
    output = 1.2020569031595942853997381615114499908
    \\ Continued fraction of zeta(3) up to 15 convergents
    >> V = contfrac(zeta(3), 15)
    output = [1, 4, 1, 18, 1, 1, 1, 4, 1, 9, 9, 2, 1, 2]
    \\ Evaluate the continued fraction
    >> eval_contfrac(V)
    output = 1479821/1231074
    \\ Decimal representation of above
    >> 1479821/1231074.
    output = 1.2020569031593551646773467720055821177
\end{lstlisting}
For 15 convergents, $$\zeta(3) \approx \dfrac{1479821}{1231074}$$ is correct up to $12$ decimal places,
where as $$\zeta(3) \approx \dfrac{461424925}{383862797}$$ which is the continued fraction approximation of $\zeta(3)$ to $20$ convergents gives the correct approximation to $18$ decimal places.
In \texttt{Mathematica}, the command to construct a continued fraction is \texttt{ContinuedFraction}, while the command to evaluate a contined fraction is \texttt{FromContinuedFraction}.


\end{document}

%% file: macros.tex
\usepackage[bottom=1in, right=1in, left=1in, top=1in]{geometry}
\usepackage{mdframed}
\usepackage{amsmath}
\allowdisplaybreaks
\usepackage{amsfonts}
\usepackage{amsthm}
\usepackage{amssymb}
\usepackage{lmodern}
\usepackage{graphicx}
\usepackage{enumerate}
\usepackage{comment}
\usepackage{hyperref}

\usepackage{xcolor}
\usepackage{enumitem}
\usepackage{tikz}
\hypersetup{pdftitle={Elliptic curves from dependency equilibria}}
\hypersetup{colorlinks=true,linkcolor=blue,anchorcolor=blue,citecolor=blue, urlcolor = blue}

\usepackage{multirow}

\newtheorem{theorem}{Theorem}[section]
\newtheorem{proposition}[theorem]{Proposition}

\newtheorem{lemma}[theorem]{Lemma}
\theoremstyle{definition}
\newtheorem{example}[theorem]{Example}
\theoremstyle{definition}
\newtheorem{definition}[theorem]{Definition}
\theoremstyle{definition}
\newtheorem{remark}[theorem]{Remark}

\newcommand{\V}{\mathcal V_X}

\usepackage{listings}
\definecolor{codegreen}{rgb}{0,0.6,0}
\definecolor{codegray}{rgb}{0.5,0.5,0.5}
\definecolor{codepurple}{rgb}{0.58,0,0.82}
\definecolor{backcolour}{rgb}{0.95,0.95,0.92}
\makeatletter
\def\@tocline#1#2#3#4#5#6#7{\relax
  \ifnum #1>\c@tocdepth 
  \else
    \par \addpenalty\@secpenalty\addvspace{#2}%
    \begingroup \hyphenpenalty\@M
    \@ifempty{#4}{%
      \@tempdima\csname r@tocindent\number#1\endcsname\relax
    }{%
      \@tempdima#4\relax
    }%
    \parindent\z@ \leftskip#3\relax \advance\leftskip\@tempdima\relax
    \rightskip\@pnumwidth plus4em \parfillskip-\@pnumwidth
    #5\leavevmode\hskip-\@tempdima
      \ifcase #1
       \or\or \hskip 1em \or \hskip 2em \else \hskip 3em \fi%
      #6\nobreak\relax
    \dotfill\hbox to\@pnumwidth{\@tocpagenum{#7}}\par
    \nobreak
    \endgroup
  \fi}
\makeatother

\lstdefinestyle{mystyle}{
    backgroundcolor=\color{white},   
    commentstyle=\color{codegreen},
    keywordstyle=\color{blue},
    numberstyle=\tiny\color{codegray},
    stringstyle=\color{codepurple},
    basicstyle=\ttfamily\footnotesize,
    breakatwhitespace=false,         
    breaklines=true,                 
    captionpos=b,                    
    keepspaces=true,                 
    numbersep=5pt,                  
    showspaces=false,                
    showstringspaces=false,
    showtabs=false,                  
    tabsize=2,
    keywords = {subst, truncate, local, return, sqrt, random, ellperiods, ellinit, print, log, List, length, listput, Vec, ellap, ellan, abs, select, exp, round, prec, imag, real, concat, write, ceil, ellfromeqn, polcoef, j, forstep, contfrac, zeta , forstep, if, my, local},
    frame = single, framexleftmargin=0pt
}

\newcommand\blfootnote[1]{
    \begingroup
    \renewcommand\thefootnote{}\footnote{#1}
    \addtocounter{footnote}{-1}
    \endgroup
}

%% file: tikzimg.tex
\tikzset{every picture/.style={line width=0.75pt}} 

\begin{tikzpicture}[x=0.75pt,y=0.75pt,yscale=-1,xscale=1]

\draw   (61,17) -- (109.16,17) -- (109.16,59.32) -- (61,59.32) -- cycle ;
\draw  [fill={rgb, 255:red, 245; green, 228; blue, 5 }  ,fill opacity=1 ] (93.98,24.28) .. controls (94.32,21.21) and (97.3,18.96) .. (100.66,19.27) .. controls (104.01,19.58) and (106.45,22.32) .. (106.12,25.39) .. controls (105.78,28.46) and (102.79,30.71) .. (99.44,30.4) .. controls (96.09,30.09) and (93.65,27.35) .. (93.98,24.28) -- cycle ;
\draw  [fill={rgb, 255:red, 11; green, 111; blue, 221 }  ,fill opacity=1 ] (64.55,50.65) .. controls (64.89,47.58) and (67.88,45.34) .. (71.23,45.64) .. controls (74.58,45.95) and (77.02,48.69) .. (76.69,51.76) .. controls (76.35,54.84) and (73.37,57.08) .. (70.01,56.77) .. controls (66.66,56.46) and (64.22,53.72) .. (64.55,50.65) -- cycle ;
\draw   (115.84,17) -- (164,17) -- (164,59.32) -- (115.84,59.32) -- cycle ;
\draw  [fill={rgb, 255:red, 194; green, 39; blue, 190 }  ,fill opacity=1 ] (93.31,50.65) .. controls (93.65,47.58) and (96.64,45.34) .. (99.99,45.64) .. controls (103.34,45.95) and (105.78,48.69) .. (105.45,51.76) .. controls (105.11,54.84) and (102.12,57.08) .. (98.77,56.77) .. controls (95.42,56.46) and (92.98,53.72) .. (93.31,50.65) -- cycle ;
\draw  [fill={rgb, 255:red, 74; green, 144; blue, 226 }  ,fill opacity=0.5 ] (117.85,52.21) .. controls (117.85,50.18) and (119.5,48.53) .. (121.53,48.53) -- (155.64,48.53) .. controls (157.67,48.53) and (159.32,50.18) .. (159.32,52.21) -- (159.32,52.21) .. controls (159.32,54.24) and (157.67,55.89) .. (155.64,55.89) -- (121.53,55.89) .. controls (119.5,55.89) and (117.85,54.24) .. (117.85,52.21) -- cycle ;
\draw  [fill={rgb, 255:red, 248; green, 231; blue, 28 }  ,fill opacity=0.5 ] (154.97,19.76) .. controls (157.37,19.76) and (159.32,21.71) .. (159.32,24.11) -- (159.32,52.21) .. controls (159.32,54.61) and (157.37,56.56) .. (154.97,56.56) -- (154.97,56.56) .. controls (152.57,56.56) and (150.62,54.61) .. (150.62,52.21) -- (150.62,24.11) .. controls (150.62,21.71) and (152.57,19.76) .. (154.97,19.76) -- cycle ;
\draw  [fill={rgb, 255:red, 189; green, 16; blue, 224 }  ,fill opacity=0.5 ] (156.14,20.83) .. controls (157.9,22.44) and (157.9,25.05) .. (156.15,26.65) -- (124.69,55.49) .. controls (122.95,57.1) and (120.11,57.09) .. (118.35,55.48) -- (118.35,55.48) .. controls (116.6,53.87) and (116.59,51.27) .. (118.34,49.67) -- (149.8,20.82) .. controls (151.55,19.22) and (154.39,19.22) .. (156.14,20.83) -- cycle ;

\draw   (239.16,17) -- (191,17) -- (191,59.32) -- (239.16,59.32) -- cycle ;
\draw  [fill={rgb, 255:red, 245; green, 228; blue, 5 }  ,fill opacity=1 ] (206.17,24.28) .. controls (205.84,21.21) and (202.85,18.96) .. (199.5,19.27) .. controls (196.15,19.58) and (193.7,22.32) .. (194.04,25.39) .. controls (194.37,28.46) and (197.36,30.71) .. (200.71,30.4) .. controls (204.06,30.09) and (206.51,27.35) .. (206.17,24.28) -- cycle ;
\draw  [fill={rgb, 255:red, 11; green, 111; blue, 221 }  ,fill opacity=1 ] (235.6,50.65) .. controls (235.27,47.58) and (232.28,45.34) .. (228.93,45.64) .. controls (225.58,45.95) and (223.13,48.69) .. (223.47,51.76) .. controls (223.8,54.84) and (226.79,57.08) .. (230.14,56.77) .. controls (233.49,56.46) and (235.94,53.72) .. (235.6,50.65) -- cycle ;
\draw  [fill={rgb, 255:red, 194; green, 39; blue, 190 }  ,fill opacity=1 ] (206.84,50.65) .. controls (206.51,47.58) and (203.52,45.34) .. (200.17,45.64) .. controls (196.82,45.95) and (194.37,48.69) .. (194.71,51.76) .. controls (195.04,54.84) and (198.03,57.08) .. (201.38,56.77) .. controls (204.73,56.46) and (207.18,53.72) .. (206.84,50.65) -- cycle ;

\draw   (246.84,59.32) -- (295,59.32) -- (295,17) -- (246.84,17) -- cycle ;
\draw  [fill={rgb, 255:red, 248; green, 231; blue, 28 }  ,fill opacity=0.5 ] (285.97,56.56) .. controls (288.37,56.56) and (290.32,54.61) .. (290.32,52.21) -- (290.32,24.11) .. controls (290.32,21.71) and (288.37,19.76) .. (285.97,19.76) -- (285.97,19.76) .. controls (283.57,19.76) and (281.62,21.71) .. (281.62,24.11) -- (281.62,52.21) .. controls (281.62,54.61) and (283.57,56.56) .. (285.97,56.56) -- cycle ;
\draw  [fill={rgb, 255:red, 189; green, 16; blue, 224 }  ,fill opacity=0.5 ] (287.14,55.48) .. controls (288.9,53.87) and (288.9,51.27) .. (287.15,49.67) -- (255.69,20.82) .. controls (253.95,19.22) and (251.11,19.22) .. (249.35,20.83) -- (249.35,20.83) .. controls (247.6,22.44) and (247.59,25.05) .. (249.34,26.65) -- (280.8,55.49) .. controls (282.55,57.1) and (285.39,57.09) .. (287.14,55.48) -- cycle ;

\draw  [fill={rgb, 255:red, 74; green, 144; blue, 226 }  ,fill opacity=0.5 ] (248.85,24.11) .. controls (248.85,26.14) and (250.5,27.79) .. (252.53,27.79) -- (286.64,27.79) .. controls (288.67,27.79) and (290.32,26.14) .. (290.32,24.11) -- (290.32,24.11) .. controls (290.32,22.07) and (288.67,20.43) .. (286.64,20.43) -- (252.53,20.43) .. controls (250.5,20.43) and (248.85,22.07) .. (248.85,24.11) -- cycle ;

\draw   (369.16,18) -- (321,18) -- (321,60.32) -- (369.16,60.32) -- cycle ;
\draw  [fill={rgb, 255:red, 245; green, 228; blue, 5 }  ,fill opacity=1 ] (336.17,25.28) .. controls (335.84,22.21) and (332.85,19.96) .. (329.5,20.27) .. controls (326.15,20.58) and (323.7,23.32) .. (324.04,26.39) .. controls (324.37,29.46) and (327.36,31.71) .. (330.71,31.4) .. controls (334.06,31.09) and (336.51,28.35) .. (336.17,25.28) -- cycle ;
\draw  [fill={rgb, 255:red, 11; green, 111; blue, 221 }  ,fill opacity=1 ] (365.6,51.65) .. controls (365.27,48.58) and (362.28,46.34) .. (358.93,46.64) .. controls (355.58,46.95) and (353.13,49.69) .. (353.47,52.76) .. controls (353.8,55.84) and (356.79,58.08) .. (360.14,57.77) .. controls (363.49,57.46) and (365.94,54.72) .. (365.6,51.65) -- cycle ;
\draw  [fill={rgb, 255:red, 194; green, 39; blue, 190 }  ,fill opacity=1 ] (365.87,26.38) .. controls (365.9,23.29) and (363.2,20.7) .. (359.84,20.61) .. controls (356.47,20.52) and (353.72,22.95) .. (353.69,26.04) .. controls (353.65,29.13) and (356.35,31.71) .. (359.71,31.8) .. controls (363.08,31.9) and (365.83,29.47) .. (365.87,26.38) -- cycle ;
\draw   (423.33,18) -- (375.18,18) -- (375.18,60.32) -- (423.33,60.32) -- cycle ;
\draw  [fill={rgb, 255:red, 248; green, 231; blue, 28 }  ,fill opacity=0.5 ] (421.32,53.21) .. controls (421.32,51.18) and (419.68,49.53) .. (417.64,49.53) -- (383.54,49.53) .. controls (381.5,49.53) and (379.86,51.18) .. (379.86,53.21) -- (379.86,53.21) .. controls (379.86,55.24) and (381.5,56.89) .. (383.54,56.89) -- (417.64,56.89) .. controls (419.68,56.89) and (421.32,55.24) .. (421.32,53.21) -- cycle ;
\draw  [fill={rgb, 255:red, 74; green, 144; blue, 226 }  ,fill opacity=0.5 ] (384.2,20.76) .. controls (381.8,20.76) and (379.86,22.71) .. (379.86,25.11) -- (379.86,53.21) .. controls (379.86,55.61) and (381.8,57.56) .. (384.2,57.56) -- (384.2,57.56) .. controls (386.61,57.56) and (388.55,55.61) .. (388.55,53.21) -- (388.55,25.11) .. controls (388.55,22.71) and (386.61,20.76) .. (384.2,20.76) -- cycle ;
\draw  [fill={rgb, 255:red, 189; green, 16; blue, 224 }  ,fill opacity=0.5 ] (383.03,21.83) .. controls (381.28,23.44) and (381.27,26.05) .. (383.02,27.65) -- (414.48,56.49) .. controls (416.23,58.1) and (419.07,58.09) .. (420.82,56.48) -- (420.82,56.48) .. controls (422.58,54.87) and (422.58,52.27) .. (420.83,50.67) -- (389.38,21.82) .. controls (387.63,20.22) and (384.79,20.22) .. (383.03,21.83) -- cycle ;

\draw   (451,18) -- (499.16,18) -- (499.16,60.32) -- (451,60.32) -- cycle ;
\draw  [fill={rgb, 255:red, 245; green, 228; blue, 5 }  ,fill opacity=1 ] (483.98,25.28) .. controls (484.32,22.21) and (487.3,19.96) .. (490.66,20.27) .. controls (494.01,20.58) and (496.45,23.32) .. (496.12,26.39) .. controls (495.78,29.46) and (492.79,31.71) .. (489.44,31.4) .. controls (486.09,31.09) and (483.65,28.35) .. (483.98,25.28) -- cycle ;
\draw  [fill={rgb, 255:red, 11; green, 111; blue, 221 }  ,fill opacity=1 ] (454.55,51.65) .. controls (454.89,48.58) and (457.88,46.34) .. (461.23,46.64) .. controls (464.58,46.95) and (467.02,49.69) .. (466.69,52.76) .. controls (466.35,55.84) and (463.37,58.08) .. (460.01,57.77) .. controls (456.66,57.46) and (454.22,54.72) .. (454.55,51.65) -- cycle ;
\draw  [fill={rgb, 255:red, 194; green, 39; blue, 190 }  ,fill opacity=1 ] (454.29,27.11) .. controls (454.39,24.02) and (457.2,21.56) .. (460.57,21.61) .. controls (463.93,21.67) and (466.58,24.22) .. (466.47,27.31) .. controls (466.37,30.4) and (463.56,32.86) .. (460.19,32.8) .. controls (456.83,32.74) and (454.18,30.19) .. (454.29,27.11) -- cycle ;
\draw   (554.33,60.32) -- (506.18,60.32) -- (506.18,18) -- (554.33,18) -- cycle ;
\draw  [fill={rgb, 255:red, 248; green, 231; blue, 28 }  ,fill opacity=0.5 ] (552.32,25.11) .. controls (552.32,27.14) and (550.68,28.79) .. (548.64,28.79) -- (514.54,28.79) .. controls (512.5,28.79) and (510.86,27.14) .. (510.86,25.11) -- (510.86,25.11) .. controls (510.86,23.07) and (512.5,21.43) .. (514.54,21.43) -- (548.64,21.43) .. controls (550.68,21.43) and (552.32,23.07) .. (552.32,25.11) -- cycle ;
\draw  [fill={rgb, 255:red, 74; green, 144; blue, 226 }  ,fill opacity=0.5 ] (515.2,57.56) .. controls (512.8,57.56) and (510.86,55.61) .. (510.86,53.21) -- (510.86,25.11) .. controls (510.86,22.71) and (512.8,20.76) .. (515.2,20.76) -- (515.2,20.76) .. controls (517.61,20.76) and (519.55,22.71) .. (519.55,25.11) -- (519.55,53.21) .. controls (519.55,55.61) and (517.61,57.56) .. (515.2,57.56) -- cycle ;
\draw  [fill={rgb, 255:red, 189; green, 16; blue, 224 }  ,fill opacity=0.5 ] (514.03,56.48) .. controls (512.28,54.87) and (512.27,52.27) .. (514.02,50.67) -- (545.48,21.82) .. controls (547.23,20.22) and (550.07,20.22) .. (551.82,21.83) -- (551.82,21.83) .. controls (553.58,23.44) and (553.58,26.05) .. (551.83,27.65) -- (520.38,56.49) .. controls (518.63,58.1) and (515.79,58.09) .. (514.03,56.48) -- cycle ;

\draw   (60,97) -- (108.16,97) -- (108.16,139.32) -- (60,139.32) -- cycle ;
\draw  [fill={rgb, 255:red, 245; green, 228; blue, 5 }  ,fill opacity=1 ] (92.98,104.28) .. controls (93.32,101.21) and (96.3,98.96) .. (99.66,99.27) .. controls (103.01,99.58) and (105.45,102.32) .. (105.12,105.39) .. controls (104.78,108.46) and (101.79,110.71) .. (98.44,110.4) .. controls (95.09,110.09) and (92.65,107.35) .. (92.98,104.28) -- cycle ;
\draw  [fill={rgb, 255:red, 11; green, 111; blue, 221 }  ,fill opacity=1 ] (63.55,130.65) .. controls (63.89,127.58) and (66.88,125.34) .. (70.23,125.64) .. controls (73.58,125.95) and (76.02,128.69) .. (75.69,131.76) .. controls (75.35,134.84) and (72.37,137.08) .. (69.01,136.77) .. controls (65.66,136.46) and (63.22,133.72) .. (63.55,130.65) -- cycle ;
\draw   (162.33,139.32) -- (114.18,139.32) -- (114.18,97) -- (162.33,97) -- cycle ;
\draw  [fill={rgb, 255:red, 74; green, 144; blue, 226 }  ,fill opacity=0.5 ] (160.32,103.75) .. controls (160.32,105.78) and (158.68,107.43) .. (156.64,107.43) -- (122.54,107.43) .. controls (120.5,107.43) and (118.86,105.78) .. (118.86,103.75) -- (118.86,103.75) .. controls (118.86,101.71) and (120.5,100.07) .. (122.54,100.07) -- (156.64,100.07) .. controls (158.68,100.07) and (160.32,101.71) .. (160.32,103.75) -- cycle ;
\draw  [fill={rgb, 255:red, 248; green, 231; blue, 28 }  ,fill opacity=0.5 ] (123.2,136.56) .. controls (120.8,136.56) and (118.86,134.61) .. (118.86,132.21) -- (118.86,104.11) .. controls (118.86,101.71) and (120.8,99.76) .. (123.2,99.76) -- (123.2,99.76) .. controls (125.61,99.76) and (127.55,101.71) .. (127.55,104.11) -- (127.55,132.21) .. controls (127.55,134.61) and (125.61,136.56) .. (123.2,136.56) -- cycle ;
\draw  [fill={rgb, 255:red, 189; green, 16; blue, 224 }  ,fill opacity=0.5 ] (122.03,135.48) .. controls (120.28,133.87) and (120.27,131.27) .. (122.02,129.67) -- (153.48,100.82) .. controls (155.23,99.22) and (158.07,99.22) .. (159.82,100.83) -- (159.82,100.83) .. controls (161.58,102.44) and (161.58,105.05) .. (159.83,106.65) -- (128.38,135.49) .. controls (126.63,137.1) and (123.79,137.09) .. (122.03,135.48) -- cycle ;
\draw  [fill={rgb, 255:red, 194; green, 39; blue, 190 }  ,fill opacity=1 ] (92.31,130.65) .. controls (92.65,127.58) and (95.64,125.34) .. (98.99,125.64) .. controls (102.34,125.95) and (104.78,128.69) .. (104.45,131.76) .. controls (104.11,134.84) and (101.12,137.08) .. (97.77,136.77) .. controls (94.42,136.46) and (91.98,133.72) .. (92.31,130.65) -- cycle ;

\draw   (240.16,98) -- (192,98) -- (192,140.32) -- (240.16,140.32) -- cycle ;
\draw  [fill={rgb, 255:red, 245; green, 228; blue, 5 }  ,fill opacity=1 ] (207.17,105.28) .. controls (206.84,102.21) and (203.85,99.96) .. (200.5,100.27) .. controls (197.15,100.58) and (194.7,103.32) .. (195.04,106.39) .. controls (195.37,109.46) and (198.36,111.71) .. (201.71,111.4) .. controls (205.06,111.09) and (207.51,108.35) .. (207.17,105.28) -- cycle ;
\draw  [fill={rgb, 255:red, 11; green, 111; blue, 221 }  ,fill opacity=1 ] (236.6,131.65) .. controls (236.27,128.58) and (233.28,126.34) .. (229.93,126.64) .. controls (226.58,126.95) and (224.13,129.69) .. (224.47,132.76) .. controls (224.8,135.84) and (227.79,138.08) .. (231.14,137.77) .. controls (234.49,137.46) and (236.94,134.72) .. (236.6,131.65) -- cycle ;
\draw  [fill={rgb, 255:red, 194; green, 39; blue, 190 }  ,fill opacity=1 ] (207.84,131.65) .. controls (207.51,128.58) and (204.52,126.34) .. (201.17,126.64) .. controls (197.82,126.95) and (195.37,129.69) .. (195.71,132.76) .. controls (196.04,135.84) and (199.03,138.08) .. (202.38,137.77) .. controls (205.73,137.46) and (208.18,134.72) .. (207.84,131.65) -- cycle ;

\draw   (295.33,98) -- (247.18,98) -- (247.18,140.32) -- (295.33,140.32) -- cycle ;
\draw  [fill={rgb, 255:red, 74; green, 144; blue, 226 }  ,fill opacity=0.5 ] (293.32,133.57) .. controls (293.32,131.54) and (291.68,129.89) .. (289.64,129.89) -- (255.54,129.89) .. controls (253.5,129.89) and (251.86,131.54) .. (251.86,133.57) -- (251.86,133.57) .. controls (251.86,135.6) and (253.5,137.25) .. (255.54,137.25) -- (289.64,137.25) .. controls (291.68,137.25) and (293.32,135.6) .. (293.32,133.57) -- cycle ;
\draw  [fill={rgb, 255:red, 248; green, 231; blue, 28 }  ,fill opacity=0.5 ] (256.2,100.76) .. controls (253.8,100.76) and (251.86,102.71) .. (251.86,105.11) -- (251.86,133.21) .. controls (251.86,135.61) and (253.8,137.56) .. (256.2,137.56) -- (256.2,137.56) .. controls (258.61,137.56) and (260.55,135.61) .. (260.55,133.21) -- (260.55,105.11) .. controls (260.55,102.71) and (258.61,100.76) .. (256.2,100.76) -- cycle ;
\draw  [fill={rgb, 255:red, 189; green, 16; blue, 224 }  ,fill opacity=0.5 ] (255.03,101.83) .. controls (253.28,103.44) and (253.27,106.05) .. (255.02,107.65) -- (286.48,136.49) .. controls (288.23,138.1) and (291.07,138.09) .. (292.82,136.48) -- (292.82,136.48) .. controls (294.58,134.87) and (294.58,132.27) .. (292.83,130.67) -- (261.38,101.82) .. controls (259.63,100.22) and (256.79,100.22) .. (255.03,101.83) -- cycle ;

\draw   (368.16,97) -- (320,97) -- (320,139.32) -- (368.16,139.32) -- cycle ;
\draw  [fill={rgb, 255:red, 245; green, 228; blue, 5 }  ,fill opacity=1 ] (335.17,104.28) .. controls (334.84,101.21) and (331.85,98.96) .. (328.5,99.27) .. controls (325.15,99.58) and (322.7,102.32) .. (323.04,105.39) .. controls (323.37,108.46) and (326.36,110.71) .. (329.71,110.4) .. controls (333.06,110.09) and (335.51,107.35) .. (335.17,104.28) -- cycle ;
\draw  [fill={rgb, 255:red, 11; green, 111; blue, 221 }  ,fill opacity=1 ] (364.6,130.65) .. controls (364.27,127.58) and (361.28,125.34) .. (357.93,125.64) .. controls (354.58,125.95) and (352.13,128.69) .. (352.47,131.76) .. controls (352.8,134.84) and (355.79,137.08) .. (359.14,136.77) .. controls (362.49,136.46) and (364.94,133.72) .. (364.6,130.65) -- cycle ;
\draw  [fill={rgb, 255:red, 194; green, 39; blue, 190 }  ,fill opacity=1 ] (364.87,105.38) .. controls (364.9,102.29) and (362.2,99.7) .. (358.84,99.61) .. controls (355.47,99.52) and (352.72,101.95) .. (352.69,105.04) .. controls (352.65,108.13) and (355.35,110.71) .. (358.71,110.8) .. controls (362.08,110.9) and (364.83,108.47) .. (364.87,105.38) -- cycle ;
\draw   (374.84,139.32) -- (423,139.32) -- (423,97) -- (374.84,97) -- cycle ;
\draw  [fill={rgb, 255:red, 248; green, 231; blue, 28 }  ,fill opacity=0.5 ] (376.85,104.11) .. controls (376.85,106.14) and (378.5,107.79) .. (380.53,107.79) -- (414.64,107.79) .. controls (416.67,107.79) and (418.32,106.14) .. (418.32,104.11) -- (418.32,104.11) .. controls (418.32,102.07) and (416.67,100.43) .. (414.64,100.43) -- (380.53,100.43) .. controls (378.5,100.43) and (376.85,102.07) .. (376.85,104.11) -- cycle ;
\draw  [fill={rgb, 255:red, 74; green, 144; blue, 226 }  ,fill opacity=0.5 ] (413.97,136.56) .. controls (416.37,136.56) and (418.32,134.61) .. (418.32,132.21) -- (418.32,104.11) .. controls (418.32,101.71) and (416.37,99.76) .. (413.97,99.76) -- (413.97,99.76) .. controls (411.57,99.76) and (409.62,101.71) .. (409.62,104.11) -- (409.62,132.21) .. controls (409.62,134.61) and (411.57,136.56) .. (413.97,136.56) -- cycle ;
\draw  [fill={rgb, 255:red, 189; green, 16; blue, 224 }  ,fill opacity=0.5 ] (415.14,135.48) .. controls (416.9,133.87) and (416.9,131.27) .. (415.15,129.67) -- (383.69,100.82) .. controls (381.95,99.22) and (379.11,99.22) .. (377.35,100.83) -- (377.35,100.83) .. controls (375.6,102.44) and (375.59,105.05) .. (377.34,106.65) -- (408.8,135.49) .. controls (410.55,137.1) and (413.39,137.09) .. (415.14,135.48) -- cycle ;

\draw   (450,97) -- (498.16,97) -- (498.16,139.32) -- (450,139.32) -- cycle ;
\draw  [fill={rgb, 255:red, 11; green, 111; blue, 221 }  ,fill opacity=1 ] (453.55,130.65) .. controls (453.89,127.58) and (456.88,125.34) .. (460.23,125.64) .. controls (463.58,125.95) and (466.02,128.69) .. (465.69,131.76) .. controls (465.35,134.84) and (462.37,137.08) .. (459.01,136.77) .. controls (455.66,136.46) and (453.22,133.72) .. (453.55,130.65) -- cycle ;
\draw   (504.84,97) -- (553,97) -- (553,139.32) -- (504.84,139.32) -- cycle ;
\draw  [fill={rgb, 255:red, 248; green, 231; blue, 28 }  ,fill opacity=0.5 ] (506.85,132.21) .. controls (506.85,130.18) and (508.5,128.53) .. (510.53,128.53) -- (544.64,128.53) .. controls (546.67,128.53) and (548.32,130.18) .. (548.32,132.21) -- (548.32,132.21) .. controls (548.32,134.24) and (546.67,135.89) .. (544.64,135.89) -- (510.53,135.89) .. controls (508.5,135.89) and (506.85,134.24) .. (506.85,132.21) -- cycle ;
\draw  [fill={rgb, 255:red, 74; green, 144; blue, 226 }  ,fill opacity=0.5 ] (543.97,99.76) .. controls (546.37,99.76) and (548.32,101.71) .. (548.32,104.11) -- (548.32,132.21) .. controls (548.32,134.61) and (546.37,136.56) .. (543.97,136.56) -- (543.97,136.56) .. controls (541.57,136.56) and (539.62,134.61) .. (539.62,132.21) -- (539.62,104.11) .. controls (539.62,101.71) and (541.57,99.76) .. (543.97,99.76) -- cycle ;
\draw  [fill={rgb, 255:red, 189; green, 16; blue, 224 }  ,fill opacity=0.5 ] (545.14,100.83) .. controls (546.9,102.44) and (546.9,105.05) .. (545.15,106.65) -- (513.69,135.49) .. controls (511.95,137.1) and (509.11,137.09) .. (507.35,135.48) -- (507.35,135.48) .. controls (505.6,133.87) and (505.59,131.27) .. (507.34,129.67) -- (538.8,100.82) .. controls (540.55,99.22) and (543.39,99.22) .. (545.14,100.83) -- cycle ;
\draw  [fill={rgb, 255:red, 194; green, 39; blue, 190 }  ,fill opacity=1 ] (453.29,105.38) .. controls (453.26,102.29) and (455.96,99.7) .. (459.32,99.61) .. controls (462.68,99.52) and (465.44,101.95) .. (465.47,105.04) .. controls (465.5,108.13) and (462.8,110.71) .. (459.44,110.8) .. controls (456.08,110.9) and (453.32,108.47) .. (453.29,105.38) -- cycle ;

\draw   (61,178) -- (109.16,178) -- (109.16,220.32) -- (61,220.32) -- cycle ;
\draw  [fill={rgb, 255:red, 245; green, 228; blue, 5 }  ,fill opacity=1 ] (93.98,185.28) .. controls (94.32,182.21) and (97.3,179.96) .. (100.66,180.27) .. controls (104.01,180.58) and (106.45,183.32) .. (106.12,186.39) .. controls (105.78,189.46) and (102.79,191.71) .. (99.44,191.4) .. controls (96.09,191.09) and (93.65,188.35) .. (93.98,185.28) -- cycle ;
\draw  [fill={rgb, 255:red, 11; green, 111; blue, 221 }  ,fill opacity=1 ] (64.55,211.65) .. controls (64.89,208.58) and (67.88,206.34) .. (71.23,206.64) .. controls (74.58,206.95) and (77.02,209.69) .. (76.69,212.76) .. controls (76.35,215.84) and (73.37,218.08) .. (70.01,217.77) .. controls (66.66,217.46) and (64.22,214.72) .. (64.55,211.65) -- cycle ;
\draw   (115.84,178) -- (164,178) -- (164,220.32) -- (115.84,220.32) -- cycle ;
\draw  [fill={rgb, 255:red, 248; green, 231; blue, 28 }  ,fill opacity=0.5 ] (117.85,213.57) .. controls (117.85,211.54) and (119.5,209.89) .. (121.53,209.89) -- (155.64,209.89) .. controls (157.67,209.89) and (159.32,211.54) .. (159.32,213.57) -- (159.32,213.57) .. controls (159.32,215.6) and (157.67,217.25) .. (155.64,217.25) -- (121.53,217.25) .. controls (119.5,217.25) and (117.85,215.6) .. (117.85,213.57) -- cycle ;
\draw  [fill={rgb, 255:red, 74; green, 144; blue, 226 }  ,fill opacity=0.5 ] (154.97,180.76) .. controls (157.37,180.76) and (159.32,182.71) .. (159.32,185.11) -- (159.32,213.21) .. controls (159.32,215.61) and (157.37,217.56) .. (154.97,217.56) -- (154.97,217.56) .. controls (152.57,217.56) and (150.62,215.61) .. (150.62,213.21) -- (150.62,185.11) .. controls (150.62,182.71) and (152.57,180.76) .. (154.97,180.76) -- cycle ;
\draw  [fill={rgb, 255:red, 189; green, 16; blue, 224 }  ,fill opacity=0.5 ] (156.14,181.83) .. controls (157.9,183.44) and (157.9,186.05) .. (156.15,187.65) -- (124.69,216.49) .. controls (122.95,218.1) and (120.11,218.09) .. (118.35,216.48) -- (118.35,216.48) .. controls (116.6,214.87) and (116.59,212.27) .. (118.34,210.67) -- (149.8,181.82) .. controls (151.55,180.22) and (154.39,180.22) .. (156.14,181.83) -- cycle ;
\draw  [fill={rgb, 255:red, 194; green, 39; blue, 190 }  ,fill opacity=1 ] (93.31,211.65) .. controls (93.65,208.58) and (96.64,206.34) .. (99.99,206.64) .. controls (103.34,206.95) and (105.78,209.69) .. (105.45,212.76) .. controls (105.11,215.84) and (102.12,218.08) .. (98.77,217.77) .. controls (95.42,217.46) and (92.98,214.72) .. (93.31,211.65) -- cycle ;

\draw   (240.16,178) -- (192,178) -- (192,220.32) -- (240.16,220.32) -- cycle ;
\draw  [fill={rgb, 255:red, 245; green, 228; blue, 5 }  ,fill opacity=1 ] (207.17,185.28) .. controls (206.84,182.21) and (203.85,179.96) .. (200.5,180.27) .. controls (197.15,180.58) and (194.7,183.32) .. (195.04,186.39) .. controls (195.37,189.46) and (198.36,191.71) .. (201.71,191.4) .. controls (205.06,191.09) and (207.51,188.35) .. (207.17,185.28) -- cycle ;
\draw  [fill={rgb, 255:red, 11; green, 111; blue, 221 }  ,fill opacity=1 ] (236.6,211.65) .. controls (236.27,208.58) and (233.28,206.34) .. (229.93,206.64) .. controls (226.58,206.95) and (224.13,209.69) .. (224.47,212.76) .. controls (224.8,215.84) and (227.79,218.08) .. (231.14,217.77) .. controls (234.49,217.46) and (236.94,214.72) .. (236.6,211.65) -- cycle ;
\draw  [fill={rgb, 255:red, 194; green, 39; blue, 190 }  ,fill opacity=1 ] (207.84,211.65) .. controls (207.51,208.58) and (204.52,206.34) .. (201.17,206.64) .. controls (197.82,206.95) and (195.37,209.69) .. (195.71,212.76) .. controls (196.04,215.84) and (199.03,218.08) .. (202.38,217.77) .. controls (205.73,217.46) and (208.18,214.72) .. (207.84,211.65) -- cycle ;

\draw   (295.33,178) -- (247.18,178) -- (247.18,220.32) -- (295.33,220.32) -- cycle ;
\draw  [fill={rgb, 255:red, 248; green, 231; blue, 28 }  ,fill opacity=0.5 ] (293.32,213.57) .. controls (293.32,211.54) and (291.68,209.89) .. (289.64,209.89) -- (255.54,209.89) .. controls (253.5,209.89) and (251.86,211.54) .. (251.86,213.57) -- (251.86,213.57) .. controls (251.86,215.6) and (253.5,217.25) .. (255.54,217.25) -- (289.64,217.25) .. controls (291.68,217.25) and (293.32,215.6) .. (293.32,213.57) -- cycle ;
\draw  [fill={rgb, 255:red, 74; green, 144; blue, 226 }  ,fill opacity=0.5 ] (256.2,180.76) .. controls (253.8,180.76) and (251.86,182.71) .. (251.86,185.11) -- (251.86,213.21) .. controls (251.86,215.61) and (253.8,217.56) .. (256.2,217.56) -- (256.2,217.56) .. controls (258.61,217.56) and (260.55,215.61) .. (260.55,213.21) -- (260.55,185.11) .. controls (260.55,182.71) and (258.61,180.76) .. (256.2,180.76) -- cycle ;
\draw  [fill={rgb, 255:red, 189; green, 16; blue, 224 }  ,fill opacity=0.5 ] (255.03,181.83) .. controls (253.28,183.44) and (253.27,186.05) .. (255.02,187.65) -- (286.48,216.49) .. controls (288.23,218.1) and (291.07,218.09) .. (292.82,216.48) -- (292.82,216.48) .. controls (294.58,214.87) and (294.58,212.27) .. (292.83,210.67) -- (261.38,181.82) .. controls (259.63,180.22) and (256.79,180.22) .. (255.03,181.83) -- cycle ;

\draw   (369.16,177) -- (321,177) -- (321,219.32) -- (369.16,219.32) -- cycle ;
\draw  [fill={rgb, 255:red, 245; green, 228; blue, 5 }  ,fill opacity=1 ] (336.17,184.28) .. controls (335.84,181.21) and (332.85,178.96) .. (329.5,179.27) .. controls (326.15,179.58) and (323.7,182.32) .. (324.04,185.39) .. controls (324.37,188.46) and (327.36,190.71) .. (330.71,190.4) .. controls (334.06,190.09) and (336.51,187.35) .. (336.17,184.28) -- cycle ;
\draw  [fill={rgb, 255:red, 11; green, 111; blue, 221 }  ,fill opacity=1 ] (365.6,210.65) .. controls (365.27,207.58) and (362.28,205.34) .. (358.93,205.64) .. controls (355.58,205.95) and (353.13,208.69) .. (353.47,211.76) .. controls (353.8,214.84) and (356.79,217.08) .. (360.14,216.77) .. controls (363.49,216.46) and (365.94,213.72) .. (365.6,210.65) -- cycle ;
\draw  [fill={rgb, 255:red, 194; green, 39; blue, 190 }  ,fill opacity=1 ] (365.87,185.38) .. controls (365.9,182.29) and (363.2,179.7) .. (359.84,179.61) .. controls (356.47,179.52) and (353.72,181.95) .. (353.69,185.04) .. controls (353.65,188.13) and (356.35,190.71) .. (359.71,190.8) .. controls (363.08,190.9) and (365.83,188.47) .. (365.87,185.38) -- cycle ;
\draw   (375.84,219.32) -- (424,219.32) -- (424,177) -- (375.84,177) -- cycle ;
\draw  [fill={rgb, 255:red, 74; green, 144; blue, 226 }  ,fill opacity=0.5 ] (377.85,184.11) .. controls (377.85,186.14) and (379.5,187.79) .. (381.53,187.79) -- (415.64,187.79) .. controls (417.67,187.79) and (419.32,186.14) .. (419.32,184.11) -- (419.32,184.11) .. controls (419.32,182.07) and (417.67,180.43) .. (415.64,180.43) -- (381.53,180.43) .. controls (379.5,180.43) and (377.85,182.07) .. (377.85,184.11) -- cycle ;
\draw  [fill={rgb, 255:red, 248; green, 231; blue, 28 }  ,fill opacity=0.5 ] (414.97,216.56) .. controls (417.37,216.56) and (419.32,214.61) .. (419.32,212.21) -- (419.32,184.11) .. controls (419.32,181.71) and (417.37,179.76) .. (414.97,179.76) -- (414.97,179.76) .. controls (412.57,179.76) and (410.62,181.71) .. (410.62,184.11) -- (410.62,212.21) .. controls (410.62,214.61) and (412.57,216.56) .. (414.97,216.56) -- cycle ;
\draw  [fill={rgb, 255:red, 189; green, 16; blue, 224 }  ,fill opacity=0.5 ] (416.14,215.48) .. controls (417.9,213.87) and (417.9,211.27) .. (416.15,209.67) -- (384.69,180.82) .. controls (382.95,179.22) and (380.11,179.22) .. (378.35,180.83) -- (378.35,180.83) .. controls (376.6,182.44) and (376.59,185.05) .. (378.34,186.65) -- (409.8,215.49) .. controls (411.55,217.1) and (414.39,217.09) .. (416.14,215.48) -- cycle ;

\draw   (451,178) -- (499.16,178) -- (499.16,220.32) -- (451,220.32) -- cycle ;
\draw  [fill={rgb, 255:red, 245; green, 228; blue, 5 }  ,fill opacity=1 ] (483.98,185.28) .. controls (484.32,182.21) and (487.3,179.96) .. (490.66,180.27) .. controls (494.01,180.58) and (496.45,183.32) .. (496.12,186.39) .. controls (495.78,189.46) and (492.79,191.71) .. (489.44,191.4) .. controls (486.09,191.09) and (483.65,188.35) .. (483.98,185.28) -- cycle ;
\draw  [fill={rgb, 255:red, 11; green, 111; blue, 221 }  ,fill opacity=1 ] (454.55,211.65) .. controls (454.89,208.58) and (457.88,206.34) .. (461.23,206.64) .. controls (464.58,206.95) and (467.02,209.69) .. (466.69,212.76) .. controls (466.35,215.84) and (463.37,218.08) .. (460.01,217.77) .. controls (456.66,217.46) and (454.22,214.72) .. (454.55,211.65) -- cycle ;
\draw  [fill={rgb, 255:red, 194; green, 39; blue, 190 }  ,fill opacity=1 ] (454.29,186.38) .. controls (454.26,183.29) and (456.96,180.7) .. (460.32,180.61) .. controls (463.68,180.52) and (466.44,182.95) .. (466.47,186.04) .. controls (466.5,189.13) and (463.8,191.71) .. (460.44,191.8) .. controls (457.08,191.9) and (454.32,189.47) .. (454.29,186.38) -- cycle ;
\draw   (553.33,220.32) -- (505.18,220.32) -- (505.18,178) -- (553.33,178) -- cycle ;
\draw  [fill={rgb, 255:red, 74; green, 144; blue, 226 }  ,fill opacity=0.5 ] (551.32,185.11) .. controls (551.32,187.14) and (549.68,188.79) .. (547.64,188.79) -- (513.54,188.79) .. controls (511.5,188.79) and (509.86,187.14) .. (509.86,185.11) -- (509.86,185.11) .. controls (509.86,183.07) and (511.5,181.43) .. (513.54,181.43) -- (547.64,181.43) .. controls (549.68,181.43) and (551.32,183.07) .. (551.32,185.11) -- cycle ;
\draw  [fill={rgb, 255:red, 248; green, 231; blue, 28 }  ,fill opacity=0.5 ] (514.2,217.56) .. controls (511.8,217.56) and (509.86,215.61) .. (509.86,213.21) -- (509.86,185.11) .. controls (509.86,182.71) and (511.8,180.76) .. (514.2,180.76) -- (514.2,180.76) .. controls (516.61,180.76) and (518.55,182.71) .. (518.55,185.11) -- (518.55,213.21) .. controls (518.55,215.61) and (516.61,217.56) .. (514.2,217.56) -- cycle ;
\draw  [fill={rgb, 255:red, 189; green, 16; blue, 224 }  ,fill opacity=0.5 ] (513.03,216.48) .. controls (511.28,214.87) and (511.27,212.27) .. (513.02,210.67) -- (544.48,181.82) .. controls (546.23,180.22) and (549.07,180.22) .. (550.82,181.83) -- (550.82,181.83) .. controls (552.58,183.44) and (552.58,186.05) .. (550.83,187.65) -- (519.38,216.49) .. controls (517.63,218.1) and (514.79,218.09) .. (513.03,216.48) -- cycle ;

\draw   (60,257) -- (108.16,257) -- (108.16,299.32) -- (60,299.32) -- cycle ;
\draw  [fill={rgb, 255:red, 245; green, 228; blue, 5 }  ,fill opacity=1 ] (92.98,264.28) .. controls (93.32,261.21) and (96.3,258.96) .. (99.66,259.27) .. controls (103.01,259.58) and (105.45,262.32) .. (105.12,265.39) .. controls (104.78,268.46) and (101.79,270.71) .. (98.44,270.4) .. controls (95.09,270.09) and (92.65,267.35) .. (92.98,264.28) -- cycle ;
\draw  [fill={rgb, 255:red, 11; green, 111; blue, 221 }  ,fill opacity=1 ] (63.55,290.65) .. controls (63.89,287.58) and (66.88,285.34) .. (70.23,285.64) .. controls (73.58,285.95) and (76.02,288.69) .. (75.69,291.76) .. controls (75.35,294.84) and (72.37,297.08) .. (69.01,296.77) .. controls (65.66,296.46) and (63.22,293.72) .. (63.55,290.65) -- cycle ;
\draw   (162.33,299.32) -- (114.18,299.32) -- (114.18,257) -- (162.33,257) -- cycle ;
\draw  [fill={rgb, 255:red, 248; green, 231; blue, 28 }  ,fill opacity=0.5 ] (160.32,263.75) .. controls (160.32,265.78) and (158.68,267.43) .. (156.64,267.43) -- (122.54,267.43) .. controls (120.5,267.43) and (118.86,265.78) .. (118.86,263.75) -- (118.86,263.75) .. controls (118.86,261.71) and (120.5,260.07) .. (122.54,260.07) -- (156.64,260.07) .. controls (158.68,260.07) and (160.32,261.71) .. (160.32,263.75) -- cycle ;
\draw  [fill={rgb, 255:red, 74; green, 144; blue, 226 }  ,fill opacity=0.5 ] (123.2,296.56) .. controls (120.8,296.56) and (118.86,294.61) .. (118.86,292.21) -- (118.86,264.11) .. controls (118.86,261.71) and (120.8,259.76) .. (123.2,259.76) -- (123.2,259.76) .. controls (125.61,259.76) and (127.55,261.71) .. (127.55,264.11) -- (127.55,292.21) .. controls (127.55,294.61) and (125.61,296.56) .. (123.2,296.56) -- cycle ;
\draw  [fill={rgb, 255:red, 189; green, 16; blue, 224 }  ,fill opacity=0.5 ] (122.03,295.48) .. controls (120.28,293.87) and (120.27,291.27) .. (122.02,289.67) -- (153.48,260.82) .. controls (155.23,259.22) and (158.07,259.22) .. (159.82,260.83) -- (159.82,260.83) .. controls (161.58,262.44) and (161.58,265.05) .. (159.83,266.65) -- (128.38,295.49) .. controls (126.63,297.1) and (123.79,297.09) .. (122.03,295.48) -- cycle ;
\draw  [fill={rgb, 255:red, 194; green, 39; blue, 190 }  ,fill opacity=1 ] (92.31,290.65) .. controls (92.65,287.58) and (95.64,285.34) .. (98.99,285.64) .. controls (102.34,285.95) and (104.78,288.69) .. (104.45,291.76) .. controls (104.11,294.84) and (101.12,297.08) .. (97.77,296.77) .. controls (94.42,296.46) and (91.98,293.72) .. (92.31,290.65) -- cycle ;

\draw   (238.16,258) -- (190,258) -- (190,300.32) -- (238.16,300.32) -- cycle ;
\draw  [fill={rgb, 255:red, 245; green, 228; blue, 5 }  ,fill opacity=1 ] (205.17,265.28) .. controls (204.84,262.21) and (201.85,259.96) .. (198.5,260.27) .. controls (195.15,260.58) and (192.7,263.32) .. (193.04,266.39) .. controls (193.37,269.46) and (196.36,271.71) .. (199.71,271.4) .. controls (203.06,271.09) and (205.51,268.35) .. (205.17,265.28) -- cycle ;
\draw  [fill={rgb, 255:red, 11; green, 111; blue, 221 }  ,fill opacity=1 ] (234.6,291.65) .. controls (234.27,288.58) and (231.28,286.34) .. (227.93,286.64) .. controls (224.58,286.95) and (222.13,289.69) .. (222.47,292.76) .. controls (222.8,295.84) and (225.79,298.08) .. (229.14,297.77) .. controls (232.49,297.46) and (234.94,294.72) .. (234.6,291.65) -- cycle ;
\draw  [fill={rgb, 255:red, 194; green, 39; blue, 190 }  ,fill opacity=1 ] (205.84,291.65) .. controls (205.51,288.58) and (202.52,286.34) .. (199.17,286.64) .. controls (195.82,286.95) and (193.37,289.69) .. (193.71,292.76) .. controls (194.04,295.84) and (197.03,298.08) .. (200.38,297.77) .. controls (203.73,297.46) and (206.18,294.72) .. (205.84,291.65) -- cycle ;

\draw   (245.84,300.32) -- (294,300.32) -- (294,258) -- (245.84,258) -- cycle ;
\draw  [fill={rgb, 255:red, 248; green, 231; blue, 28 }  ,fill opacity=0.5 ] (247.85,264.75) .. controls (247.85,266.78) and (249.5,268.43) .. (251.53,268.43) -- (285.64,268.43) .. controls (287.67,268.43) and (289.32,266.78) .. (289.32,264.75) -- (289.32,264.75) .. controls (289.32,262.71) and (287.67,261.07) .. (285.64,261.07) -- (251.53,261.07) .. controls (249.5,261.07) and (247.85,262.71) .. (247.85,264.75) -- cycle ;
\draw  [fill={rgb, 255:red, 74; green, 144; blue, 226 }  ,fill opacity=0.5 ] (284.97,297.56) .. controls (287.37,297.56) and (289.32,295.61) .. (289.32,293.21) -- (289.32,265.11) .. controls (289.32,262.71) and (287.37,260.76) .. (284.97,260.76) -- (284.97,260.76) .. controls (282.57,260.76) and (280.62,262.71) .. (280.62,265.11) -- (280.62,293.21) .. controls (280.62,295.61) and (282.57,297.56) .. (284.97,297.56) -- cycle ;
\draw  [fill={rgb, 255:red, 189; green, 16; blue, 224 }  ,fill opacity=0.5 ] (286.14,296.48) .. controls (287.9,294.87) and (287.9,292.27) .. (286.15,290.67) -- (254.69,261.82) .. controls (252.95,260.22) and (250.11,260.22) .. (248.35,261.83) -- (248.35,261.83) .. controls (246.6,263.44) and (246.59,266.05) .. (248.34,267.65) -- (279.8,296.49) .. controls (281.55,298.1) and (284.39,298.09) .. (286.14,296.48) -- cycle ;

\draw   (368.16,258) -- (320,258) -- (320,300.32) -- (368.16,300.32) -- cycle ;
\draw  [fill={rgb, 255:red, 245; green, 228; blue, 5 }  ,fill opacity=1 ] (335.17,265.28) .. controls (334.84,262.21) and (331.85,259.96) .. (328.5,260.27) .. controls (325.15,260.58) and (322.7,263.32) .. (323.04,266.39) .. controls (323.37,269.46) and (326.36,271.71) .. (329.71,271.4) .. controls (333.06,271.09) and (335.51,268.35) .. (335.17,265.28) -- cycle ;
\draw  [fill={rgb, 255:red, 11; green, 111; blue, 221 }  ,fill opacity=1 ] (364.6,291.65) .. controls (364.27,288.58) and (361.28,286.34) .. (357.93,286.64) .. controls (354.58,286.95) and (352.13,289.69) .. (352.47,292.76) .. controls (352.8,295.84) and (355.79,298.08) .. (359.14,297.77) .. controls (362.49,297.46) and (364.94,294.72) .. (364.6,291.65) -- cycle ;
\draw  [fill={rgb, 255:red, 194; green, 39; blue, 190 }  ,fill opacity=1 ] (364.87,266.38) .. controls (364.9,263.29) and (362.2,260.7) .. (358.84,260.61) .. controls (355.47,260.52) and (352.72,262.95) .. (352.69,266.04) .. controls (352.65,269.13) and (355.35,271.71) .. (358.71,271.8) .. controls (362.08,271.9) and (364.83,269.47) .. (364.87,266.38) -- cycle ;
\draw   (422.33,258) -- (374.18,258) -- (374.18,300.32) -- (422.33,300.32) -- cycle ;
\draw  [fill={rgb, 255:red, 74; green, 144; blue, 226 }  ,fill opacity=0.5 ] (420.32,293.21) .. controls (420.32,291.18) and (418.68,289.53) .. (416.64,289.53) -- (382.54,289.53) .. controls (380.5,289.53) and (378.86,291.18) .. (378.86,293.21) -- (378.86,293.21) .. controls (378.86,295.24) and (380.5,296.89) .. (382.54,296.89) -- (416.64,296.89) .. controls (418.68,296.89) and (420.32,295.24) .. (420.32,293.21) -- cycle ;
\draw  [fill={rgb, 255:red, 248; green, 231; blue, 28 }  ,fill opacity=0.5 ] (383.2,260.76) .. controls (380.8,260.76) and (378.86,262.71) .. (378.86,265.11) -- (378.86,293.21) .. controls (378.86,295.61) and (380.8,297.56) .. (383.2,297.56) -- (383.2,297.56) .. controls (385.61,297.56) and (387.55,295.61) .. (387.55,293.21) -- (387.55,265.11) .. controls (387.55,262.71) and (385.61,260.76) .. (383.2,260.76) -- cycle ;
\draw  [fill={rgb, 255:red, 189; green, 16; blue, 224 }  ,fill opacity=0.5 ] (382.03,261.83) .. controls (380.28,263.44) and (380.27,266.05) .. (382.02,267.65) -- (413.48,296.49) .. controls (415.23,298.1) and (418.07,298.09) .. (419.82,296.48) -- (419.82,296.48) .. controls (421.58,294.87) and (421.58,292.27) .. (419.83,290.67) -- (388.38,261.82) .. controls (386.63,260.22) and (383.79,260.22) .. (382.03,261.83) -- cycle ;

\draw   (451,258) -- (499.16,258) -- (499.16,300.32) -- (451,300.32) -- cycle ;
\draw  [fill={rgb, 255:red, 245; green, 228; blue, 5 }  ,fill opacity=1 ] (483.98,265.28) .. controls (484.32,262.21) and (487.3,259.96) .. (490.66,260.27) .. controls (494.01,260.58) and (496.45,263.32) .. (496.12,266.39) .. controls (495.78,269.46) and (492.79,271.71) .. (489.44,271.4) .. controls (486.09,271.09) and (483.65,268.35) .. (483.98,265.28) -- cycle ;
\draw  [fill={rgb, 255:red, 11; green, 111; blue, 221 }  ,fill opacity=1 ] (454.55,291.65) .. controls (454.89,288.58) and (457.88,286.34) .. (461.23,286.64) .. controls (464.58,286.95) and (467.02,289.69) .. (466.69,292.76) .. controls (466.35,295.84) and (463.37,298.08) .. (460.01,297.77) .. controls (456.66,297.46) and (454.22,294.72) .. (454.55,291.65) -- cycle ;
\draw  [fill={rgb, 255:red, 194; green, 39; blue, 190 }  ,fill opacity=1 ] (454.29,266.38) .. controls (454.26,263.29) and (456.96,260.7) .. (460.32,260.61) .. controls (463.68,260.52) and (466.44,262.95) .. (466.47,266.04) .. controls (466.5,269.13) and (463.8,271.71) .. (460.44,271.8) .. controls (457.08,271.9) and (454.32,269.47) .. (454.29,266.38) -- cycle ;
\draw   (505.84,258) -- (554,258) -- (554,300.32) -- (505.84,300.32) -- cycle ;
\draw  [fill={rgb, 255:red, 74; green, 144; blue, 226 }  ,fill opacity=0.5 ] (507.85,293.21) .. controls (507.85,291.18) and (509.5,289.53) .. (511.53,289.53) -- (545.64,289.53) .. controls (547.67,289.53) and (549.32,291.18) .. (549.32,293.21) -- (549.32,293.21) .. controls (549.32,295.24) and (547.67,296.89) .. (545.64,296.89) -- (511.53,296.89) .. controls (509.5,296.89) and (507.85,295.24) .. (507.85,293.21) -- cycle ;
\draw  [fill={rgb, 255:red, 248; green, 231; blue, 28 }  ,fill opacity=0.5 ] (544.97,260.76) .. controls (547.37,260.76) and (549.32,262.71) .. (549.32,265.11) -- (549.32,293.21) .. controls (549.32,295.61) and (547.37,297.56) .. (544.97,297.56) -- (544.97,297.56) .. controls (542.57,297.56) and (540.62,295.61) .. (540.62,293.21) -- (540.62,265.11) .. controls (540.62,262.71) and (542.57,260.76) .. (544.97,260.76) -- cycle ;
\draw  [fill={rgb, 255:red, 189; green, 16; blue, 224 }  ,fill opacity=0.5 ] (546.14,261.83) .. controls (547.9,263.44) and (547.9,266.05) .. (546.15,267.65) -- (514.69,296.49) .. controls (512.95,298.1) and (510.11,298.09) .. (508.35,296.48) -- (508.35,296.48) .. controls (506.6,294.87) and (506.59,292.27) .. (508.34,290.67) -- (539.8,261.82) .. controls (541.55,260.22) and (544.39,260.22) .. (546.14,261.83) -- cycle ;

\draw  [fill={rgb, 255:red, 245; green, 228; blue, 5 }  ,fill opacity=1 ] (482.98,104.28) .. controls (483.32,101.21) and (486.3,98.96) .. (489.66,99.27) .. controls (493.01,99.58) and (495.45,102.32) .. (495.12,105.39) .. controls (494.78,108.46) and (491.79,110.71) .. (488.44,110.4) .. controls (485.09,110.09) and (482.65,107.35) .. (482.98,104.28) -- cycle ;

\draw (229.12,305.99) node [anchor=north west][inner sep=0.75pt]  [font=\footnotesize]  {$f_{12}^{( 12)}$};
\draw (231.12,225.99) node [anchor=north west][inner sep=0.75pt]  [font=\footnotesize]  {$f_{12}^{( 11)}$};
\draw (231.12,145.99) node [anchor=north west][inner sep=0.75pt]  [font=\footnotesize]  {$f_{12}^{( 10)}$};
\draw (230.12,64.99) node [anchor=north west][inner sep=0.75pt]  [font=\footnotesize]  {$f_{12}^{( 9)}$};
\draw (100.12,64.99) node [anchor=north west][inner sep=0.75pt]  [font=\footnotesize]  {$f_{11}^{( 9)}$};
\draw (99.12,144.99) node [anchor=north west][inner sep=0.75pt]  [font=\footnotesize]  {$f_{11}^{( 10)}$};
\draw (100.12,225.99) node [anchor=north west][inner sep=0.75pt]  [font=\footnotesize]  {$f_{11}^{( 11)}$};
\draw (99.12,304.99) node [anchor=north west][inner sep=0.75pt]  [font=\footnotesize]  {$f_{11}^{( 12)}$};
\draw (359.12,144.99) node [anchor=north west][inner sep=0.75pt]  [font=\footnotesize]  {$f_{21}^{( 10)}$};
\draw (360.12,224.99) node [anchor=north west][inner sep=0.75pt]  [font=\footnotesize]  {$f_{21}^{( 11)}$};
\draw (359.12,305.99) node [anchor=north west][inner sep=0.75pt]  [font=\footnotesize]  {$f_{21}^{( 12)}$};
\draw (490.12,65.99) node [anchor=north west][inner sep=0.75pt]  [font=\footnotesize]  {$f_{22}^{( 9)}$};
\draw (360.12,65.99) node [anchor=north west][inner sep=0.75pt]  [font=\footnotesize]  {$f_{21}^{( 9)}$};
\draw (13,28.4) node [anchor=north west][inner sep=0.75pt]    {$( 9)$};
\draw (9,107.4) node [anchor=north west][inner sep=0.75pt]    {$( 10)$};
\draw (489.12,144.99) node [anchor=north west][inner sep=0.75pt]  [font=\footnotesize]  {$f_{22}^{( 10)}$};
\draw (9,190.4) node [anchor=north west][inner sep=0.75pt]    {$( 11)$};
\draw (490.12,225.99) node [anchor=north west][inner sep=0.75pt]  [font=\footnotesize]  {$f_{22}^{( 11)}$};
\draw (10,268.4) node [anchor=north west][inner sep=0.75pt]    {$( 12)$};
\draw (490.12,305.99) node [anchor=north west][inner sep=0.75pt]  [font=\footnotesize]  {$f_{22}^{( 12)}$};

\end{tikzpicture}